\documentclass[reqno,10pt, centertags]{amsart}
\usepackage{amssymb,upref,esint}

\usepackage{hyperref}
\newcommand*{\mailto}[1]{\href{mailto:#1}{\nolinkurl{#1}}}

\makeatletter
\def\theequation{\@arabic\c@equation}

\newcommand{\bbN}{{\mathbb{N}}}
\newcommand{\bbR}{{\mathbb{R}}}

\newcommand{\bbZ}{{\mathbb{Z}}}
\newcommand{\bbC}{{\mathbb{C}}}

\newcommand{\cB}{{\mathcal B}}

\newcommand{\cD}{{\mathcal D}}

\newcommand{\cF}{{\mathcal F}}
\newcommand{\cH}{{\mathcal H}}

\newcommand{\cK}{{\mathcal K}}

\newcommand{\cN}{{\mathcal N}}

\newcommand{\no}{\nonumber}
\newcommand{\lb}{\label}
\newcommand{\f}{\frac}

\newcommand{\ol}{\overline}

\newcommand{\wti}{\widetilde}

\newcommand{\Oh}{O}

\newcommand{\tr}{\text{\rm{tr}}}

\newcommand{\ran}{\text{\rm{ran}}}
\newcommand{\ind}{\text{\rm{ind}}}
\newcommand{\de}{\text{\rm{def}}}
\newcommand{\dom}{\text{\rm{dom}}}

\newcommand{\bi}{\bibitem}

\numberwithin{equation}{section}

\newtheorem{theorem}{Theorem}[section]
\newtheorem{lemma}[theorem]{Lemma}

\newtheorem{definition}[theorem]{Definition}
\newtheorem{hypothesis}[theorem]{Hypothesis}
\theoremstyle{remark}
\newtheorem{remark}[theorem]{Remark}

\begin{document}

\title[On the Index of Meromorphic Operator-Valued Functions]{On the Index of Meromorphic 
Operator-Valued Functions and Some Applications}

\author[J.\ Behrndt]{Jussi Behrndt}  
\address{Institut f\"ur Numerische Mathematik, Technische Universit\"at 
Graz, Steyrergasse 30, 8010 Graz, Austria}  
\email{\mailto{behrndt@tugraz.at}}
\urladdr{\url{http://www.math.tugraz.at/~behrndt/}}

\author[F.\ Gesztesy]{Fritz Gesztesy} 
\address{Department of Mathematics,
University of Missouri, Columbia, MO 65211, USA}
\email{\mailto{gesztesyf@missouri.edu}}
\urladdr{\url{https://www.math.missouri.edu/people/gesztesy}}

\author[H.\ Holden]{Helge Holden}
\address{Department of Mathematical Sciences,
Norwegian University of
Science and Technology, NO--7491 Trondheim, Norway}
\email{\mailto{holden@math.ntnu.no}}
\urladdr{\href{http://www.math.ntnu.no/~holden/}{http://www.math.ntnu.no/\~{}holden/}}

\author[R.\ Nichols]{Roger Nichols}
\address{Mathematics Department, The University of Tennessee at Chattanooga, 
415 EMCS Building, Dept. 6956, 615 McCallie Ave, Chattanooga, TN 37403, USA}
\email{\mailto{Roger-Nichols@utc.edu}}
\urladdr{\url{http://www.utc.edu/faculty/roger-nichols/index.php}}

\dedicatory{Dedicated with great pleasure to Pavel Exner at the occasion of his 
70th birthday.}
\date{\today}
\thanks{ J.B.\ and F.G.\ gratefully acknowledge support by the Austrian Science Fund (FWF), project
P 25162-N26. F.G.\ and H.H.\ were supported in part by the project ``Waves and Nonlinear Phenomena (WaNP)'' from the Research Council of Norway. 
R.N.\ gratefully acknowledges support from a UTC College of Arts and Sciences RCA Grant.} 
\thanks{To appear in {\it Functional Analysis and Operator Theory for Quantum Physics}, 
J.\ Dittrich, H.\ Kovarik, and A.\ Laptev (eds.), EMS Publishing House, EMS, ETH--Z\"urich, Switzerland.}
\subjclass[2010]{Primary: 47A10, 47A53, 47A56, 47A75. Secondary: 47A55, 47B10.}
\keywords{Factorization of operator-valued analytic functions, multiplicity of eigenvalues, 
index computations for finitely meromorphic operator-valued functions, Birman--Schwinger 
operators, dual pairs.}

\begin{abstract}
We revisit and connect several notions of algebraic multiplicities of zeros of analytic 
operator-valued functions and discuss the concept of the index of meromorphic 
operator-valued functions in complex, separable Hilbert spaces. Applications to abstract 
perturbation theory and associated Birman--Schwinger-type operators and to the 
operator-valued Weyl--Titchmarsh functions associated to closed extensions of dual pairs 
of closed operators are provided. 
\end{abstract}

\maketitle


{\scriptsize{\tableofcontents}}

\section{Introduction} \lb{s1}

{\it We dedicate this paper with great pleasure to Pavel Exner, whose tireless efforts as an 
ambassador for Mathematical Physics have led him to nearly every corner of this globe. 
Happy Birthday, Pavel, we hope our modest contribution to operator and spectral theory 
will cause some joy.} 

\smallskip 

The purpose of this paper is fourfold: 

$\bullet$ First, to recall recent results on factorizations of analytic 
operator-valued Fredholm functions following Howland \cite{Ho71} and more recently, 
\cite{GHN15}. 

$\bullet$ Second, apply this to algebraic multiplicities of bounded, analytic operator-valued Fredholm functions.

$\bullet$ Third, discuss the notion of an index of meromorphic operator-valued functions. 

$\bullet$ Fourth, apply this to Birman--Schwinger operators in connection with abstract perturbation theory and to operator-valued Weyl--Titchmarsh functions associated to closed extensions of dual 
pairs of closed operators.  

More precisely, in Section \ref{s2} we recall the notion of finitely-meromorphic $\cB(\cH)$-valued 
functions and some of their basic properties, state the analytic Fredholm theorem, and recall in Theorems \ref{t2.5} and \ref{t2.6} a 
factorization of analytic operator-valued Fredholm functions originally due to Howland \cite{Ho71} 
and recently revisited under somewhat more general hypotheses in \cite{GHN15}.  

Section \ref{s3} recalls the notion of zeros of finite-type of bounded, analytic operator-valued 
functions $A(\cdot)$, revisits the algebraic multiplicity \eqref{3.6j} of a zero of finite-type of 
$A(\cdot)$, relates the latter to the operator-valued argument principle (i.e., an 
operator Rouch\'e-type Theorem) and to appropriate traces of contour integrals, and finally proves equality of this notion of multiplicity with the multiplicity notion \eqref{2.31j} originally introduced by Howland \cite{Ho71} in Theorem \ref{t3.4a}, the principal result of this section. 

The topic of meromorphic operator-valued functions and the notion of their index is the principal 
subject of Section \ref{s4}. In particular, we revisit the notion of $\cB(\cH)$-valued finitely 
meromorphic functions $M(\cdot)$, introduce the notion of their index via the operator-valued argument 
principle and taking the trace of a contour integral as in \eqref{4.3j}, and finally recall the meromorphic Fredholm theorem.    

Abstract perturbation theory and applications to Birman--Schwinger-type operators $K(\cdot)$ 
are treated in Section \ref{s5}. This should be viewed as a refinement of recent results of this 
genre in \cite[Sect.~5]{GHN15}. Following Kato \cite{Ka66}, Konno and Kuroda \cite{KK66}, and Howland \cite{Ho70}, we recall a class of factorable non-self-adjoint perturbations of a 
given unperturbed non-self-adjoint operator $H_0$, giving rise to an operator $H$ as refined in 
\cite{GLMZ05} (cf.\ Theorem \ref{t5.2}), and then prove analogs of Weinstein--Aronszajn formulas, relating the difference of the algebraic multiplicity of 
an eigenvalue of $H$ and $H_0$ to the index of 
the meromorphic operator-valued function $I - K(\cdot)$ in Theorem \ref{t5.5}. 
 
Our final Section \ref{s6} focuses on closed extensions $A_0$, $A_{\Theta}$ ($\Theta$ an 
appropriate bounded operator parameter), associated to dual pairs $\{A,B\}$ of operators and their associated Weyl--Titchmarsh functions $M(\cdot)$, following work of Malamud, Mogilevskii, and Hassi \cite{MM02}, \cite{MM03}, \cite{MMH12}. Our principal new result, Theorem \ref{t6.4}, relates the difference of the algebraic multiplicity of a discrete eigenvalue of $A_{\Theta}$ and $A_0$ to the 
index of the meromorphic operator-valued function $\Theta - M(\cdot)$. 

Next, we summarize the basic notation used in this paper: Let $\cH$ and $\cK$
be separable complex Hilbert spaces, $(\,\cdot\,,\,\cdot\,)_{\cH}$ and
$(\,\cdot\,,\,\cdot\,)_{\cK}$ the scalar products in $\cH$ and $\cK$ (linear in
the second factor), and
$I_{\cH}$ and $I_{\cK}$ the identity operators in $\cH$ and $\cK$,
respectively. Next, let $T$ be a closed linear operator from
$\dom(T)\subseteq\cH$ to $\ran(T)\subseteq\cK$, with $\dom(T)$
and $\ran(T)$ denoting the domain and range of $T$. The closure of a
closable operator $S$ is denoted by $\ol S$. The kernel (null space) of $T$
is denoted by $\ker(T)$. The spectrum, point spectrum, and resolvent set of a closed linear 
operator in $\cH$ will be denoted by $\sigma(\cdot)$, $\sigma_p(\cdot)$, and $\rho(\cdot)$; the 
discrete spectrum of $T$ (i.e., points in $\sigma_p(T)$ which are isolated from the rest of 
$\sigma(T)$, and which are eigenvalues of $T$ of finite algebraic multiplicity) is 
abbreviated by $\sigma_d(T)$. The {\it algebraic multiplicity} $m_a(z_0; T)$ of an eigenvalue 
$z_0\in\sigma_d(T)$ is the dimension of the range of the corresponding {\it Riesz projection} 
$P(z_0;T)$,  
\begin{equation}
m_a(z_0; T) = \dim(\ran(P(z_0;T))) = \tr_{\cH}(P(z_0;T)),
\end{equation}
where (with the symbol $ \ointctrclockwise$ denoting 
contour integrals)
\begin{equation}
 P(z_0;T)=\f{-1}{2\pi i} \ointctrclockwise_{C(z_0;\varepsilon)} d\zeta \, 
(T - \zeta I_{\cH})^{-1}, 
\end{equation}
for $0 < \varepsilon<\varepsilon_0$ and  $D(z_0;\varepsilon_0) \backslash \{z_0\}\subset \rho(T)$; here
$D(z_0; r_0) \subset \bbC$ is the open disk with center $z_0$ and radius 
$r_0 > 0$, and $C(z_0; r_0) = \partial D(z_0; r_0)$ the corresponding circle. 
The  {\it geometric multiplicity} $m_g(z_0; T)$ of an eigenvalue  
$z_0 \in \sigma_p(T)$ is defined by 
\begin{equation}
m_g(z_0; T) = \dim(\ker((T - z_0 I_{\cH}))).  
\end{equation}
The essential spectrum of $T$ is defined by $\sigma_{ess}(T) = \sigma(T)\backslash \sigma_d(T)$.

The Banach spaces of bounded
and compact linear operators in $\cH$ are denoted by $\cB(\cH)$ and
$\cB_\infty(\cH)$, respectively. Similarly, the Schatten--von Neumann
(trace) ideals will subsequently be denoted by $\cB_p(\cH)$,
$p \in [1,\infty)$, and the subspace of all finite rank operators in $\cB_1(\cH)$ will be 
abbreviated by $\cF(\cH)$. Analogous notation $\cB(\cH_1,\cH_2)$,
$\cB_\infty (\cH_1,\cH_2)$, etc., will be used for bounded, compact, etc.,
operators between two Hilbert spaces $\cH_1$ and $\cH_2$. In addition,
$\tr_{\cH}(T)$ denotes the trace of a trace class operator $T\in\cB_1(\cH)$.

The set of bounded Fredholm operators on $\cH$ (i.e., the set of operators $T \in \cB(\cH)$ 
such that $\dim(\ker(T)) < \infty$, $\ran(T)$ is closed in $\cH$, and $\dim(\ker(T^*)) < \infty$) 
is denoted by the symbol $\Phi(\cH)$. The corresponding (Fredholm) index 
of $T \in \Phi(\cH)$ is then given by $\ind(T) = \dim(\ker(T)) - \dim(\ker(T^*))$. For a linear 
operator $S$ in $\cH$ with closed range one defines the 
{\it defect of $S$}, denoted by $\de(S)$, by the codimension of $\ran(S)$ in 
$\cH$, that is, 
\begin{equation}
\de(S) = \dim\big(\ran(S)^{\bot}\big). 
\end{equation}

The symbol $\dotplus$ denotes a direct (but not necessary orthogonal direct) decomposition 
in connection with subspaces of Banach spaces. Finally, we find it convenient to abbreviate 
$\bbN_0 = \bbN \cup \{0\}$.

\section{On Factorizations of Analytic Operator-Valued Functions} \lb{s2}

In this section, we recall factorizations of bounded, analytic operator-valued Fredholm functions following Howland \cite{Ho71} and more recently, \cite{GHN15}.

Assuming $\Omega \subseteq \bbC$  to be open and $M(\cdot)$ to be a $\cB(\cH)$-valued 
meromorphic function on $\Omega$ that has the norm convergent Laurent expansion around 
$z_0 \in \Omega$ of the type 
\begin{align}
\begin{split} 
M(z) = \sum_{k= - N_0}^\infty (z - z_0)^{k} M_k(z_0), \quad M_k(z_0) \in \cB(\cH), \, k \in \bbZ, 
\; k \geq - N_0,&       \\  
 0 < |z - z_0| < \varepsilon_0,&      \lb{2.1}
 \end{split} 
\end{align}
for some $N_0=N_0(z_0) \in \bbN$ and some $0 < \varepsilon_0=\varepsilon_0(z_0)$ sufficiently 
small, we denote the principal part, ${\rm pp}_{z_0} \, \{M(\cdot)\}$, of $M(\cdot)$ at $z_0$ by
\begin{align} 
& {\rm pp}_{z_0} \, \{M(z)\} = \sum_{k= - N_0}^{-1} (z - z_0)^{k} M_k(z_0), \quad 
M_k(z_0) \in \cB(\cH), \;  -N_0 \leq k \leq -1,    \no \\
& \hspace*{8.2cm} 0 < |z - z_0| < \varepsilon_0.    \lb{2.2}
\end{align}

Given the notation \eqref{2.2}, we start with the following definition.

\begin{definition} \label{d2.1}
Let $\Omega\subseteq\bbC$ be open and connected. Suppose that $M(\cdot)$ is a 
$\cB(\cH)$-valued analytic function on $\Omega$ except for isolated singularities in a neighborhood of which it is meromorphic. Then  
$M(\cdot)$ is called {\it finitely meromorphic at $z_0 \in\Omega$} if $M(\cdot)$ is 
analytic on the punctured disk $D(z_0;\varepsilon_0)\backslash\{z_0\} \subset \Omega$ 
centered at $z_0$ with sufficiently small $\varepsilon_0 > 0$, and the principal part of $M(\cdot)$ at $z_0$ is of finite rank, that is, the principal part of $M(\cdot)$ is of the type \eqref{2.2}, and one has 
\begin{equation}\label{2.3j}
M_k(z_0) \in \cF(\cH), \quad -N_0 \leq k \leq -1. 
\end{equation}
In addition, $M(\cdot)$ is called {\it finitely meromorphic on $\Omega$} if it is meromorphic 
on $\Omega$ and finitely meromorphic at each of its poles.
\end{definition}

In this context, we mention the following useful result:

\begin{lemma} [{\cite[Lemma\ XI.9.3]{GGK90}}, {\cite[Proposition\ 4.2.2]{GL09}}] \label{l2.2} 
Let $\Omega\subseteq\bbC$ be open and connected and $M_j(\cdot)$, $j=1,2$, be 
$\mathcal B(\mathcal H)$-valued 
finitely meromorphic functions at $z_0 \in \Omega$. Then 
$M_1(\cdot) M_2(\cdot)$ and $M_2(\cdot) M_1(\cdot)$ are finitely meromorphic at 
$z_0 \in \Omega$, and for $0 <\varepsilon< \varepsilon_0$ sufficiently small, 
\begin{equation}\label{2.4j1}
\ointctrclockwise_{C(z_0;\varepsilon)} d \zeta \, M_1(\zeta) M_2(\zeta)   \in \cF(\cH)
\end{equation}
and 
\begin{equation}\label{2.4j2}
\ointctrclockwise_{C(z_0;\varepsilon)} d \zeta \, M_2(\zeta) M_1(\zeta) \in \cF(\cH),
\end{equation}
and the identity
\begin{equation}\label{2.5j}
 {\tr}_{\cH} \bigg(\ointctrclockwise_{C(z_0;\varepsilon)} d \zeta \, 
M_1(\zeta) M_2(\zeta)\bigg) 
= {\tr}_{\cH} \bigg(\ointctrclockwise_{C(z_0;\varepsilon)} d \zeta \, 
M_2(\zeta) M_1(\zeta)\bigg)
\end{equation}
holds. Moreover, for $0 < |z-z_0| < \varepsilon_0$ one has
\begin{equation}\label{2.6j}
 {\tr}_{\cH} \big({\rm pp}_{z_0} \, \{M_1(z) M_2(z)\}\big) = 
{\tr}_{\cH} \big({\rm pp}_{z_0} \, \{M_2(z) M_1(z)\}\big).
\end{equation}
\end{lemma}

For the remainder of this section we make the following assumptions:

\begin{hypothesis} \lb{h2.3} 
Let $\Omega \subseteq \bbC$ be open and connected, suppose that 
$A:\Omega \to \cB(\cH)$ is analytic and that 
\begin{equation}
A(z) \in \Phi(\cH) \, \text{ for all } \, z\in\Omega.    \lb{2.8j} 
\end{equation}  
\end{hypothesis}

One then recalls the analytic Fredholm theorem in the following form:

\begin{theorem} [{\cite[Sect.~4.1]{GL09}}, \cite{GS71}, \cite{Ho70}, {\cite[Thm.~VI.14]{RS80}}, \cite{St68}] \label{t2.4}  ${}$ \\
Assume that $A:\Omega \to \cB(\cH)$ satisfies Hypothesis \ref{h2.3}. Then either \\[1mm] 
$(i)$ $A(z)$ is not boundedly invertible for any $z \in \Omega$, \\[1mm]
or else, \\[1mm]
$(ii)$ $A(\cdot)^{-1}$ is finitely meromorphic on $\Omega$. More precisely, there exists a discrete 
subset $\cD_1 \subset \Omega$ $($possibly, $\cD_1 = \emptyset$$)$ such that 
$A(z)^{-1} \in \cB(\cH)$ $($and hence lies in $\Phi(\cH)$$)$ for all $z \in \Omega\backslash\cD_1$, $A(\cdot)^{-1}$ is analytic on 
$\Omega\backslash\cD_1$, meromorphic on $\Omega$, and if $z_1 \in \cD_1$ then 
\begin{equation}\label{2.9j}
A(z)^{-1} = \sum_{k= - N_0}^{\infty} (z - z_1)^{k} C_k(z_1), \quad 
0 < |z - z_1| < \varepsilon_0,    
\end{equation}
for some $N_0=N_0(z_1) \in \bbN$ and some $0 < \varepsilon_0=\varepsilon_0(z_1)$ sufficiently 
small,
with  
\begin{equation}\lb{2.10j}  
C_k(z_1) \in \cF(\cH), \; -N_0 \leq k \leq -1,  \quad  C_k(z_1) \in \cB(\cH), \; k \in \bbN_0. 
\end{equation}
In addition, 
\begin{equation}\label{2.10jfredholm}
  C_0(z_1) \in \Phi(\cH).
\end{equation}
Finally, if $[I_{\cH} - A(z)] \in \cB_{\infty}(\cH)$ for all $z \in \Omega$, then 
\begin{equation}\label{2.11j}
\big[I_{\cH} - A(z)^{-1}\big] \in \cB_{\infty}(\cH), \; z \in \Omega\backslash\cD_1, \quad
[I_{\cH} - C_0(z_1)] \in \cB_{\infty}(\cH),\; z_1 \in \cD_1.
\end{equation}
\end{theorem}
  
The following fundamental results are due to Howland \cite{Ho71} (see also 
\cite{GHN15} for more general hypotheses, replacing Howland's assumption that 
$[A(\cdot) - I_{\cH}] \in \cB_{\infty}(\cH)$ by the assumption that $A(\cdot)$ is 
Fredholm): 

\begin{theorem} [\cite{Ho71}] \lb{t2.5}
Assume that $A:\Omega \to \cB(\cH)$ satisfies Hypothesis \ref{h2.3}, suppose that $A(z)$ is boundedly invertible for 
some $z \in \Omega$ $($i.e., case $(ii)$ in Theorem \ref{t2.4} applies\,$)$, and let 
$z_0 \in \Omega$ be a pole of $A(\cdot)^{-1}$ of order $n_0 \in \bbN$. Denote by $Q_1$ any 
projection onto $\ran(A(z_0))$ and let $P_1 = I_{\cH} - Q_1$. Then,
\begin{equation}
A(z) = [Q_1 - (z-z_0) P_1] A_1(z), \quad z \in \Omega,     \lb{2.12j} 
\end{equation}
where 
\begin{align} 
& \text{$A_1(\cdot)$ is analytic on $\Omega$,} \lb{2.13j}\\
& A_1(z) \in \Phi(\cH), \quad z\in\Omega,    \lb{2.14j} \\ 
& \ind(A(z)) = \ind(A_1(z)) =0, \quad z \in \Omega, \; |z - z_0| \, \text{ sufficiently small,}   \lb{2.15j} \\ 
& \de(A_1(z_0)) \leq \de(A(z_0)),    \lb{2.16j} \\
& \text{$z_0$ is a pole of $A_1(\cdot)^{-1}$ of order $n_0 - 1$.}    \lb{2.17j}
\end{align}
In particular, $A_1:\Omega \rightarrow \cB(\cH)$ satisfies Hypothesis \ref{h2.3}. Finally, 
\begin{align}
& [I_{\cH} - A(\cdot)] \in \cF(\cH) \, \text{ $($resp., $\cB_p(\cH)$ for some $1 \leq p \leq \infty$$)$}  \lb{2.18j} \\ 
& \quad \text{if and only if } \, 
[I_{\cH} - A_1(\cdot)] \in \cF(\cH) \, \text{ $($resp., $\cB_p(\cH)$ for some $1 \leq p \leq \infty$$)$.}   
\no 
\end{align}
\end{theorem}

Assume that $A: \Omega \to \cB(\cH)$ satisfies Hypothesis \ref{h2.3} and that 
$A(\cdot)^{-1}$ has a pole at $z_0 \in \Omega$. 
The Riesz projection $P(z)$ associated with $A(z)$ and $z$ in a sufficiently small neighborhood $\cN(z_0) \subset \Omega$ of $z_0$ is defined by 
\begin{equation} \lb{2.19j}
P(z)=\f{-1}{2\pi i} \ointctrclockwise_{C(0;\varepsilon)} d\zeta \, (A(z) -
\zeta I_{\cH})^{-1}, \quad z \in \cN(z_0),
\end{equation}
where $0 < \varepsilon< \varepsilon_0$ sufficiently small (cf., e.g., \cite[Sect.\ III.6]{Ka80}). It follows that $P(\cdot)$ is analytic on $\cN(z_0)$ and 
\begin{equation}\label{2.20j}
\dim(\ran(P(z))) < \infty, \quad z \in \cN(z_0). 
\end{equation}
In addition, introduce the projections 
\begin{equation}\label{2.21j}
Q(z)=I_{\cH} - P(z), \quad z \in \cN(z_0),
\end{equation}
and the transformations (cf.\ \cite{Wo52})
\begin{equation}\label{2.22j}
T(z)=P(z_0) P(z)+Q(z_0) Q(z), \quad z \in \cN(z_0). 
\end{equation}
It follows that $T(\cdot)$ is analytic on $\cN(z_0)$ and for $|z-z_0|$ sufficiently small, also $T(\cdot)^{-1}$ is analytic, 
\begin{equation}\label{2.24j}
T(z)= I_{\cH} + \Oh(z-z_0), \quad |z - z_0| \, \text{ sufficiently small},     
\end{equation}
and without loss of generality we may assume in the following that $T(\cdot)$ and $T(\cdot)^{-1}$ 
are analytic on $\cN(z_0)$. This yields the decomposition of $\cH$ into
\begin{equation}\label{2.25j}
\cH = P(z_0) \cH \dotplus Q(z_0) \cH    
\end{equation}
and the associated $2 \times 2$ block operator decomposition of $T(z) A(z) T(z)^{-1}$ of the form
\begin{equation}\label{2.26j}
T(z) A(z) T(z)^{-1} = \begin{pmatrix} F(z) & 0 \\ 0 & G(z) \end{pmatrix}, \quad z \in \cN(z_0), 
\end{equation}
where $F(\cdot)$ and $G(\cdot)$ are analytic on $\cN(z_0)$, and, again without loss of generality, 
$G(\cdot)$ is boundedly invertible on $\cN(z_0)$,  
\begin{equation}\label{2.27j}
G(z)^{-1} \in\cB(Q(z_0)\cH),  \quad z \in \cN(z_0).     
\end{equation}   
Given the block decomposition \eqref{2.26j}, we follow Howland in introducing the quantity 
$\nu(z_0; A(\cdot))$ by 
\begin{equation}\label{2.31j}
\nu(z_0; A(\cdot)) = m\bigl(z_0; {\det}_{\ran(P(z_0))} (F(\cdot))\bigr).
\end{equation}
Here $m(z;h)$ denotes the multiplicity function associated to a meromorphic function $h\colon\Omega\to\bbC\cup\{\infty\}$, which is  
defined by
\begin{equation}\label{2.32j}
 m(z;h)=\begin{cases} k, & \text{if $z$ is a zero of $h$ of order $k$,} \\
-k, & \text{if $z$ is a pole of order $k$,} \\
0, & \text{otherwise,} \end{cases}
\end{equation}
if $m$ does not vanish identically on $\Omega$, and by $m(z;h)=\infty$ otherwise. In the former 
case, 
\begin{equation} 
m(z;h) = \f{1}{2\pi i}\ointctrclockwise_{C(z; \varepsilon) } d\zeta \,
\f{h'(\zeta)}{h(\zeta)}, \quad z\in\Omega,      \lb{2.3.2k}
\end{equation} 
where the circle $C(z; \varepsilon)$ is chosen sufficiently small such
that $C(z; \varepsilon)$ contains no other singularities or zeros of $h$ except, possibly, $z$. 

In the present context,
since $F(\cdot)$ is analytic on $\cN(z_0)$, so is ${\det}_{\ran(P(z_0))} (F(\cdot))$, and hence 
\begin{equation}\label{2.33j}
\nu(z_0; A(\cdot)) \in \bbN_0 \, 
\text{ if ${\det}_{\ran(P(z_0))} (F(\cdot)) \not\equiv 0$ on $\cN(z_0)$.}
\end{equation}

Repeated applications of Theorem \ref{t2.5} then yields the following principal factorization result of \cite{Ho71} (again, extended to the case of Fredholm 
operators $A(\cdot)$):

\begin{theorem} [\cite{Ho71}] \lb{t2.6}
Assume that $A:\Omega \to \cB(\cH)$ satisfies Hypothesis \ref{h2.3}, suppose that $A(z)$ is boundedly invertible for 
some $z \in \Omega$ $($i.e., case $(ii)$ in Theorem \ref{t2.4} applies\,$)$, and let 
$z_0 \in \Omega$ be a pole of $A(\cdot)^{-1}$ of order $n_0 \in \bbN$. Then there exist projections $P_j$ and $Q_j = I_{\cH} - P_j$ in $\cH$ such that with 
$p_j = \dim(\ran(P_j))$, $1 \leq j \leq n_0$, one infers that 
\begin{equation}\lb{2.34j}
A(z) = [Q_1 - (z - z_0) P_1] [Q_2 - (z - z_0) P_2] \cdots [Q_{n_0} - (z - z_0) P_{n_0}] A_{n_0}(z), 
\quad z \in \Omega,    
\end{equation}
and 
\begin{equation}\lb{2.35j}
1 \leq p_{n_0} \leq p_{n_0 -1} \leq \cdots \leq p_2 \leq p_1 < \infty,    
\end{equation}
where
\begin{align} 
& \text{$A_{n_0}(\cdot)$ is analytic on $\Omega$,}    \lb{2.36j} \\
& A_{n_0}(z) \in \Phi(\cH), \quad z\in\Omega,      \lb{2.37j} \\ 
& \ind(A(z)) = \ind(A_{n_0}(z)) = 0, \quad z \in \Omega, \; |z - z_0| \, \text{ sufficiently small,}     \lb{2.38j} \\
& A_{n_0}(z)^{-1} \in \cB(\cH), \quad z \in \Omega, \; |z - z_0| \, \text{ sufficiently small.}      \lb{2.39j} 
\end{align}
In addition, 
\begin{equation}\lb{2.40j}
p_1 = \dim(\ker(A(z_0)) = m_g(0; A(z_0)),   
\end{equation}
and hence
\begin{equation}\lb{2.41j}
\nu(z_0; A(\cdot)) = \sum_{j=1}^{n_0} p_j \geq m_g(0; A(z_0)), \quad \nu(z_0; A(\cdot)) \geq n_0,  
\end{equation}
and, in particular, $z_0$ is a simple pole of $A(\cdot)^{-1}$ if and only if 
\begin{equation}\lb{2.42j}
\nu(z_0; A(\cdot)) = m_g(0; A(z_0)).
\end{equation}
Finally, 
\begin{align}\label{2.43j}
& [I_{\cH} - A(\cdot)] \in \cF(\cH) \, \text{ $($resp., $\cB_p(\cH)$ for some $1 \leq p \leq \infty$$)$}   \\ 
& \quad \text{if and only if } \, 
[I_{\cH} - A_{n_0}(\cdot)] \in \cF(\cH) \, \text{ $($resp., $\cB_p(\cH)$ for some $1 \leq p \leq \infty$$)$.}  \no 
\end{align}
\end{theorem}

We refer to \cite{GHN15} for analogous factorizations as in Theorems 
\ref{t2.5} and \ref{t2.6} but with the order of factors in \eqref{2.12j} and \eqref{2.34j} interchanged.

\section{Algebraic Multiplicities of Zeros of Analytic Fredholm Operators} \lb{s3}

In this section we recall algebraic multiplicities of zeros of analytic Fredholm operators 
following \cite{GHN15} and relate this to Howland's notion  in \eqref{2.31j}.
The pertinent facts in this context can be found in \cite{GS71} (see 
also, \cite[Sects.\ XI.8, XI.9]{GGK90}, \cite[Ch.\ 4]{GL09}, and \cite[Sect.\ 11]{Ma88}). 
We follow the presentation in \cite{GHN15}.

First the notion of zeros of finite-type is recalled.

\begin{definition} \label{def3.1}
Let $\Omega \subseteq \bbC$ be open and connected, $z_0 \in \Omega$, and suppose that 
$A:\Omega \to \cB(\cH)$ is analytic on $\Omega$. Then $z_0$ is called a {\it zero of finite-type 
of $A(\cdot)$} if $A(z_0)\in \Phi(\cH)$ is a Fredholm operator, $\ker(A(z_0)) \neq \{0\}$, and $A(\cdot)$ is 
boundedly invertible on $D(z_0; \varepsilon_0) \backslash \{z_0\}$, for some sufficiently small 
$\varepsilon_0 > 0$. 
\end{definition}

Assume that $A:\Omega \to \cB(\cH)$ is analytic on $\Omega$ and that
$z_0$ is a zero of finite-type of $A(\cdot)$. Since $A(\cdot)$ is 
boundedly invertible on $D(z_0; \varepsilon_0) \backslash \{z_0\}$, for sufficiently small 
$\varepsilon_0 > 0$, it follows that 
\begin{equation}\label{3.5j}
\ind(A(z_0)) =\dim(\ker(A(z_0))) - \dim(\ker(A(z_0)^*))= 0,  
\end{equation}
and hence by 
\cite{GS71} (or by 
\cite[Theorem\ XI.8.1]{GGK90}) there exists a neighborhood 
$\cN(z_0) \subset \Omega$ and analytic and boundedly invertible 
operator-valued functions $E_j: \Omega \to \cB(\cH)$, $j=1,2$, such that 
\begin{equation}\label{3.2j}
A(z) = E_1(z) \wti A(z) E_2(z), \quad z \in \cN(z_0), 
\end{equation}
where $\wti A(\cdot)$ is of the form
\begin{equation}\label{3.3j}
\wti A(z) = \wti P_0 + \sum_{j=1}^r (z - z_0)^{n_j} \wti P_j, \quad z \in \cN(z_0),   
\end{equation}
with 
\begin{align} 
& \wti P_k, \; 0 \leq k \leq r, \, \text{ mutually disjoint projections in $\cH$,}\no    \\ 
& \big[I_{\cH} - \wti P_0 \big] \in \cF(\cH), \quad 
\dim\big(\ran\big(\wti P_j\big)\big) = 1, \quad 1 \leq j \leq r,      \lb{3.3jj} \\ 
& n_1 \leq n_2 \leq \dots \leq n_r, \quad n_j \in \bbN, \; 1 \leq j \leq r.    \no 
\end{align} 
The integers $n_j$, $1\leq j \leq r$, in \eqref{3.3jj} are 
uniquely determined by $A(\cdot)$, and the geometric multiplicity 
$m_g(0; A(z_0))$ of the eigenvalue $0$ of $A(z_0)$ is given by
\begin{equation}\label{3.4j}
m_g(0; A(z_0)) = \dim\big(\ran\big(I_{\cH} - \wti P_0\big)\big). 
\end{equation}

The following definition can be found in \cite[Sect.\ XI.9]{GGK90}, \cite{GS71}.

\begin{definition} \label{def3.2}
Let $\Omega \subseteq \bbC$ be open and connected, $z_0 \in \Omega$,  suppose that 
$A:\Omega \to \cB(\cH)$ is analytic on $\Omega$, and assume that $z_0$ is a zero of finite-type 
of $A(\cdot)$.  Then $m_a(z_0; A(\cdot))$, the {\it algebraic multiplicity of the zero of $A(\cdot)$ 
at $z_0$}, is defined to be 
\begin{equation}\lb{3.6j}
m_a(z_0; A(\cdot)) = \sum_{j=1}^r n_j,   
\end{equation} 
with $n_j$, $1 \leq j \leq r$, introduced in \eqref{3.3jj}. 
\end{definition}

Let $A:\Omega \to \cB(\cH)$ be analytic on $\Omega$ and assume that $z_0$ is a zero of finite-type 
of $A(\cdot)$. As shown in \cite[Theorem\ XI.9.1]{GGK90}, \cite{GS71} one has an extension of the argument principle for 
scalar analytic functions to the operator-valued case  in the form 
\begin{equation}\label{3.10j}
\begin{split}
 m_a(z_0; A(\cdot)) &= {\tr}_{\cH}\bigg(\f{1}{2\pi i} \ointctrclockwise_{C(z_0; \varepsilon)} d\zeta \, A'(\zeta) A(\zeta)^{-1}\bigg)  \\
                    &= {\tr}_{\cH}\bigg(\f{1}{2\pi i} \ointctrclockwise_{C(z_0; \varepsilon)} d\zeta \, A(\zeta)^{-1} A'(\zeta)\bigg) 
\end{split}
\end{equation}
for $0 < \varepsilon<\varepsilon_0$ sufficiently small as 
in Definition \ref{def3.1}.
Since $A(\cdot)^{-1}$ is finitely meromorphic by Theorem~\ref{t2.4}, the integrals in \eqref{3.10j} 
are finite rank operators (the analytic and non-finite-rank part under the integral in \eqref{3.10j} yielding a zero contribution when integrated over 
$C(z_0; \varepsilon)$) and hence the trace in \eqref{3.10j} is 
well-defined. 
Next, recalling our notation of the principal part of an operator-valued meromorphic function in 
\eqref{2.2}, one also obtains 
\begin{equation}\label{3.11j}
 \begin{split}
  m_a(z_0; A(\cdot))& = {\tr}_{\cH}\bigg(\f{1}{2\pi i}  \ointctrclockwise_{C(z_0; \varepsilon)} d\zeta \, {\rm pp}_{z_0} \, \big\{A'(\zeta) A(\zeta)^{-1}\big\}\bigg) \\
                    & = {\tr}_{\cH}\bigg(\f{1}{2\pi i}  \ointctrclockwise_{C(z_0; \varepsilon)} d\zeta \, {\rm pp}_{z_0} \, \big\{A(\zeta)^{-1}A'(\zeta) \big\}\bigg). 
 \end{split}
\end{equation}

Note that in the special case where $A(z) = A - z I_{\cH}$, $z \in \Omega$, one has from \eqref{3.10j}
\begin{equation}\label{special}
\begin{split}
 m_a(z_0; A(\cdot))&= {\tr}_{\cH}\bigg(\f{1}{2\pi i} 
\ointctrclockwise_{C(z_0; \varepsilon)} d\zeta \, 
A'(\zeta) A(\zeta)^{-1}\bigg)\\
&= {\tr}_{\cH}\bigg(\f{-1}{2\pi i} 
\ointctrclockwise_{C(z_0; \varepsilon)} d\zeta \,   (A-\zeta I_\cH)^{-1}\bigg)  \\
&= m_a(z_0; A).
\end{split}
\end{equation}
However, in general the algebraic multiplicity 
$m_a(z_0; A(\cdot))$ of a zero of $A(\cdot)$ at $z_0$ must be distinguished from the algebraic multiplicity $m_a(0; A(z_0))$ of the eigenvalue $0$ of the operator $A(z_0)$. 

We conclude this section with the connection between the algebraic multiplicity 
$m_a(z_0; A(\cdot))$ of a zero of $A(\cdot)$ at $z_0$ in Definition~\ref{def3.2} and Howland's notion 
of multiplicity $\nu(z_0; A(\cdot))$ in \eqref{2.31j}. Note that if $A:\Omega \to \cB(\cH)$ is analytic on $\Omega$ and 
$z_0$ is a zero of finite-type  then Hypothesis~\ref{h2.3} is automatically satisfied on a sufficiently small open neighborhood of $z_0$ and hence
the quantity $\nu(z_0; A(\cdot))$ is well defined.

\begin{theorem} \label{t3.4a}
Assume that  $z_0$ is a zero of 
finite-type of $A(\cdot)$.~Then the algebraic multiplicity 
$m_a(z_0; A(\cdot))$ of the zero of $A(\cdot)$ at $z_0$ and the quantity $\nu(z_0; A(\cdot))$ coincide, that is,  
\begin{equation}\label{3.xxj}
m_a(z_0; A(\cdot)) =  \nu(z_0; A(\cdot)).
\end{equation}
\end{theorem} 
\begin{proof}
Without loss of generality we may assume that $z_0=0$ for the remainder of the proof of 
Theorem \ref{t3.4a}. 
According to \eqref{3.10j} we then have
\begin{equation}\label{3.15j}
 m_a(0; A(\cdot)) = {\tr}_{\cH} \bigg(\f{1}{2 \pi i} 
\ointctrclockwise_{C(0; \varepsilon)} d\zeta \, A(\zeta)^{-1} A'(\zeta) \bigg)
\end{equation}
for $0<\varepsilon<\varepsilon_0$ sufficiently small.
An application of 
Theorem~\ref{t2.6} (using the notation employed in the latter) yields 
\begin{equation}\label{j1}
A(z) = [Q_1 - z P_1][Q_2 - z P_2] \cdots 
[Q_{n_0} - z P_{n_0}] A_{n_0} (z), \quad z\in D(0;\varepsilon_0),
\end{equation}
and 
\begin{equation}\label{j2}
\nu(0; A(\cdot)) = \sum_{j=1}^{n_0} p_j,  
\end{equation}
where 
\begin{equation}\label{j2-2}
p_j = \dim(\ran(P_j))\quad\text{and} \quad   
Q_j = I_{\cH} - P_j, \; 1 \leq j \leq n_0.   
\end{equation}

In the following we compute the trace of the integral in \eqref{3.15j}. For this one notes that
by \eqref{j1}
\begin{equation}\label{j3}
 A(z)^{-1}=[A_{n_0} (z)]^{-1} [Q_{n_0} - z P_{n_0}]^{-1} \cdots [Q_1 - z P_1]^{-1},  
\end{equation}
and 
\begin{equation}\label{j4}
\begin{split}
 A'(z)&=[-P_1][Q_2 - z P_2] \cdots [Q_{n_0} - z P_{n_0}] 
A_{n_0} (z)    \\
& \qquad + [Q_1 - z P_1] [-P_2] \cdots [Q_{n_0} - z P_{n_0}] A_{n_0} (z) \\
& \qquad + \cdots + \\
& \qquad + 
[Q_1 - z P_1] [Q_2 - z P_2] \cdots [Q_{n_0 - 1} - z P_{n_0 - 1}] 
[- P_{n_0}] A_{n_0} (z)    \\
& \qquad + [Q_1 - z P_1] [Q_2 - zP_2] \cdots [Q_{n_0} - z P_{n_0}] 
A_{n_0}' (z).
 \end{split}
\end{equation}
Hence one obtains 
\begin{equation}\label{j55}
 \begin{split}
  A(z)^{-1}A'(z)&= [A_{n_0} (z)]^{-1} [Q_{n_0} - z P_{n_0}]^{-1} \cdots [Q_1 - z P_1]^{-1}     \\
& \,\,\,\quad \times  \big\{[-P_1][Q_2 - z P_2] \cdots [Q_{n_0} - z P_{n_0}]  \\
& \quad\qquad + [Q_1 - z P_1] [-P_2] \cdots [Q_{n_0} - z P_{n_0}]   \\
& \quad\qquad + \cdots + \\
& \quad\qquad +
[Q_1 - z P_1] \cdots [Q_{n_0 - 1} - z P_{n_0 - 1}] 
[- P_{n_0}] \big\} A_{n_0} (z)\\
& \,\,\,\quad + [A_{n_0} (z)]^{-1} A_{n_0}' (z), 
 \end{split}
\end{equation}
and since the last term on the right-hand side of \eqref{j55} is analytic at $z_0=0$, its contour 
integral over $C(0;\varepsilon)$, 
$0 < \varepsilon < \varepsilon_0$, vanishes,
\begin{align}\label{j6}
& \ointctrclockwise_{C(0; \varepsilon)} d\zeta \,  A(\zeta)^{-1}A'(\zeta)  
 = \ointctrclockwise_{C(0; \varepsilon)} d\zeta 
 [A_{n_0} (\zeta)]^{-1} [Q_{n_0} - \zeta P_{n_0}]^{-1} 
\cdots [Q_1 - \zeta P_1]^{-1}    \no \\
& \qquad \times 
\big\{[-P_1][Q_2 - \zeta P_2] \cdots [Q_{n_0} - \zeta P_{n_0}] 
+ [Q_1 - \zeta P_1] [-P_2] \cdots [Q_{n_0} - \zeta P_{n_0}]    \no \\
& \qquad \quad + \cdots +
[Q_1 - \zeta P_1] \cdots [Q_{n_0 - 1} - \zeta P_{n_0 - 1}] 
[- P_{n_0}] \big\} A_{n_0} (\zeta).    
\end{align}
Now one obtains from \eqref{j6} upon 
repeatedly applying cyclicity of the trace (i.e., 
${\tr}_{\cH}(CD) = {\tr}_{\cH}(DC)$ for $C, D \in \cB(\cH)$, with 
$CD, DC \in \cB_1(\cH)$),
\begin{align} \label{j7}
& \ointctrclockwise_{C(0; \varepsilon)} d\zeta \,  {\tr}_{\cH}\bigl(A(\zeta)^{-1} A'(\zeta)\bigr)  
 = \ointctrclockwise_{C(0; \varepsilon)} d\zeta \,{\tr}_{\cH}\Big( [Q_{n_0} - \zeta P_{n_0}]^{-1} 
\cdots [Q_1 - \zeta P_1]^{-1}    \no \\
& \qquad \times 
\big\{[-P_1][Q_2 - \zeta P_2] \cdots [Q_{n_0} - \zeta P_{n_0}] 
+ [Q_1 - \zeta P_1] [-P_2] \cdots [Q_{n_0} - \zeta P_{n_0}]     \no \\
& \qquad \quad + \cdots + 
[Q_1 - \zeta P_1] \cdots [Q_{n_0 - 1} - \zeta P_{n_0 - 1}] 
[- P_{n_0}] \big\} \Big)   \no \\
& \quad =
\ointctrclockwise_{C(0; \varepsilon)} d\zeta \, 
\sum_{j=1}^{n_0} \, {\tr}_{\cH}\big([Q_j - \zeta P_j]^{-1} [- P_j]\big)  
\end{align}
and since 
\begin{equation}\label{j8}
[Q_j - \zeta P_j]^{-1} [- P_j]=[Q_j - \zeta^{-1} P_j] [- P_j]=\zeta^{-1} P_j, 
\end{equation}
one concludes from \eqref{3.15j}, \eqref{j7}, \eqref{j8}, \eqref{j2-2}, and \eqref{j2} that
\begin{equation}\label{j9}
 \begin{split}
  m_a(0;  A(\cdot))
  & = {\tr}_{\cH} \bigg(\f{1}{2 \pi i} \ointctrclockwise_{C(0; \varepsilon)} d\zeta \, A(\zeta)^{-1} A'(\zeta) \bigg) \\
   & = \f{1}{2 \pi i} \ointctrclockwise_{C(0; \varepsilon)} d\zeta \, {\tr}_{\cH}\bigl( A(\zeta)^{-1} A'(\zeta) \bigr) \\
& = \f{1}{2 \pi i} 
\ointctrclockwise_{C(0; \varepsilon)} d\zeta \, 
\sum_{j=1}^{n_0} \, {\tr}_{\cH}\big([Q_j - \zeta P_j]^{-1} [- P_j]\big)      \\
& = \f{1}{2 \pi i} 
\ointctrclockwise_{C(0; \varepsilon)} d\zeta \, 
\bigg(\sum_{j=1}^{n_0} \, {\tr}_{\cH} (P_j) \bigg) \zeta^{-1}     \\ 
& = \sum_{j=1}^{n_0} p_j \\ 
&= \nu(0;A(\cdot)).
\end{split}
\end{equation}
\end{proof}

\section{On the Notion of an Index of Meromorphic Operator-Valued Functions} \lb{s4}

In this section we recall the notion of the index of meromorphic operator functions and the meromorphic Fredholm theorem.

\begin{hypothesis} \lb{h4.1} 
Let $\Omega \subseteq \bbC$ be open and connected and assume that $M(\cdot)$ is a $\cB(\cH)$-valued finitely meromorphic function on $\Omega$, that is,
there is a discrete set $\cD_0 \subset \Omega$ 
(i.e., a set without limit points in $\Omega$) such that $M:\Omega \backslash \cD_0 \to \cB(\cH)$ is analytic and 
for all $z_0 \in \cD_0$ one has
\begin{equation}\label{4.1j}
M(z) = \sum_{k=-N_0}^{\infty} (z-z_0)^k M_k(z_0), \quad 
0 < |z - z_0| < \varepsilon_0,   
\end{equation}
for some $N_0=N_0(z_0) \in \bbN$ and some $0 < \varepsilon_0=\varepsilon_0(z_0)$ 
sufficiently small, with  
\begin{equation}\lb{4.2j} 
M_k(z_0) \in \cF(\cH), \; -N_0 \leq k \leq -1,  \quad  M_k(z_0) \in \cB(\cH),  \; k \in \bbN_0. 
\end{equation} 
\end{hypothesis}

One observes that if $M(\cdot)$ is finitely meromorphic on $\Omega$, then also the function 
$M'(\cdot)$ is a $\cB(\cH)$-valued finitely meromorphic function on $\Omega$.
It follows from Lemma~\ref{l2.2} that the notion of the index of $M(\cdot)$ in the next definition is 
well-defined. 

\begin{definition} \lb{d4.2} 
Assume Hypothesis \ref{h4.1}, let $z_0\in \Omega$ and suppose that $M(\cdot)$ is boundedly 
invertible on 
$D(z_0; \varepsilon_0) \backslash \{z_0\}$ for some $0 < \varepsilon_0$ sufficiently small.
Assume, in addition, that the function $M(\cdot)^{-1}$ is finitely meromorphic on $D(z_0; \varepsilon_0)$. Then the {\it index of $M(\cdot)$ with respect to the counterclockwise 
oriented circle $C(z_0; \varepsilon)$}, $\ind_{C(z_0; \varepsilon)}(M(\cdot))$, is defined by 
 \begin{align}
\begin{split}
\ind_{C(z_0; \varepsilon)}(M(\cdot)) &= {\tr}_{\cH}\bigg(\f{1}{2\pi i} 
\ointctrclockwise_{C(z_0; \varepsilon)} d\zeta \, 
M'(\zeta) M(\zeta)^{-1}\bigg)     \lb{4.3j} \\
& = {\tr}_{\cH}\bigg(\f{1}{2\pi i} 
\ointctrclockwise_{C(z_0; \varepsilon)} d\zeta \, 
M(\zeta)^{-1} M'(\zeta)\bigg), \quad 0 < \varepsilon < \varepsilon_0. 
\end{split}
\end{align}  
\end{definition}

We note that this notion of an index is a bit more general than the one employed in
\cite[Ch.~4]{GL09}, \cite{GS71} and hence it is not {\it a priori} clear if the right-hand side 
of \eqref{4.3j} is
an integer. However, in the special case depicted in Theorem \ref{t3.2}\,$(ii)$ (see also \eqref{4.9}) under the additional
Hypothesis \ref{h4.3}, and in the applications in the following sections, the index indeed turns out to be an integer. For the notion of a generalized index of unbounded meromorphic operator-valued functions and its applications to Dirichlet-to-Neumann maps and abstract Weyl--Titchmarsh $M$-functions we refer to \cite{BGHN16}.

We also note that in the special case of an analytic function $M:\Omega \to \cB(\cH)$ and 
$z_0$ a zero of finite-type 
of $M(\cdot)$, it follows from Theorem~\ref{t2.4} that $M(\cdot)^{-1}$ is finitely meromorphic on 
$D(z_0; \varepsilon_0)$ for some $0 < \varepsilon_0$ sufficiently small.
Therefore, \eqref{3.10j} implies that the index of $M(\cdot)$ in \eqref{4.3j} coincides with the algebraic multiplicity of the zero of $M(\cdot)$ at $z_0$, 
\begin{equation}\label{coincide}
\ind_{C(z_0; \varepsilon)}(M(\cdot))= m_a(z_0; M(\cdot)).
\end{equation}

Moreover, if $M_j(\cdot)$, $j=1,2$, are $\cB(\cH)$-valued finitely meromorphic functions that are boundedly invertible on 
$D(z_0; \varepsilon_0) \backslash \{z_0\}$ for some $z_0\in\Omega$ and  $0 < \varepsilon_0$ sufficiently small, and $M_j(\cdot)^{-1}$, $j=1,2$, are finitely meromorphic on 
$D(z_0; \varepsilon_0) \backslash \{z_0\}$, then employing the identity
\begin{align}
\begin{split} 
[M_1(z) M_2(z)]' [M_1(z) M_2(z)]^{-1} &= M_1'(z) M_1(z)^{-1}   \\
& \quad + M_1(z) [M_2'(z) M_2(z)^{-1}] M_1(z)^{-1},   \lb{4.10}
\end{split}
\end{align} 
and taking the trace on either side yields the familiar formula
\begin{equation}
\ind_{C(z_0; \varepsilon)}(M_1(\cdot) M_2(\cdot)) = \ind_{C(z_0; \varepsilon)}(M_1(\cdot)) 
+ \ind_{C(z_0; \varepsilon)}(M_2(\cdot)),   \lb{4.11} 
\end{equation}
in particular,
\begin{equation}
\ind_{C(z_0; \varepsilon)}\big(M(\cdot)^{-1}\big) = - \ind_{C(z_0; \varepsilon)}(M(\cdot)).  
\lb{4.12} 
\end{equation}
For interesting applications of this circle of ideas see also \cite{AKL09}, \cite{BES12}, 
\cite{BBR14}, \cite{Si89}. 

Next we strengthen Hypothesis \ref{h4.1} as follows: 

\begin{hypothesis} \lb{h4.3} 
Suppose $M(\cdot)$ satisfies Hypothesis \ref{h4.1}
and assume that for every $z_0\in\cD_0$ the operator $M_0(z_0)$ in the Laurent series in \eqref{4.1j}
is a Fredholm operator, that is,
\begin{equation}
 M_0(z_0)\in \Phi(\cH),\qquad z_0\in\cD_0.
\end{equation}
In addition, suppose that 
\begin{equation}
M(z) \in \Phi(\cH), \quad z \in \Omega \backslash \cD_0. 
\end{equation}
\end{hypothesis}

One then recalls the meromorphic Fredholm theorem in the following form:

\begin{theorem} [\cite{GS71}, \cite{Ho70}, {\cite[Theorem\ XIII.13]{RS78}}, \cite{RV69}, 
\cite{St68}] \lb{t3.2}  ${}$ \\
Assume that $M(\cdot)$ satisfies Hypothesis \ref{h4.3}. Then either \\[1mm] 
$(i)$ $M(z)$ is not boundedly invertible for any $z \in \Omega\backslash\cD_0$, \\[1mm]
or else, \\[1mm]
$(ii)$ $M(\cdot)^{-1}$ is finitely meromorphic on $\Omega$. More precisely, there exists a discrete 
subset $\cD_1 \subset \Omega$ $($possibly, $\cD_1 = \emptyset$$)$ such that 
$M(z)^{-1} \in \cB(\cH)$ for all $z \in \Omega\backslash\{\cD_0 \cup \cD_1\}$, $M(\cdot)^{-1}$ extends to 
an analytic function on $\Omega\backslash\cD_1$, meromorphic on $\Omega$ such that
\begin{equation}\lb{4.5a}
M(z)^{-1} \in \Phi(\cH) \, \text{ for all } \, z \in \Omega\backslash\cD_1, 
\end{equation}
and if $z_1 \in \cD_1$, then 
\begin{equation}
M(z)^{-1} = \sum_{k= - N_0}^{\infty} (z - z_1)^{k} D_k(z_1), \quad 
0 < |z - z_1| < \varepsilon_0,    \lb{4.5}
\end{equation}
for some $N_0=N_0(z_1) \in \bbN$ and some $0 < \varepsilon_0=\varepsilon_0(z_1)$ 
sufficiently small, 
with  
\begin{equation} \lb{4.5b} 
D_k(z_1) \in \cF(\cH), \; -N_0 \leq k \leq -1,  \quad   D_k(z_1) \in \cB(\cH), \; k \in \bbN_0. 
\end{equation}
In addition,
\begin{equation}\label{jussi4fred}
 D_0(z_1) \in \Phi(\cH).
\end{equation}
Finally, if $[I_{\cH} - M(z)] \in \cB_{\infty}(\cH)$ for all $z \in \Omega\backslash\cD_0$, then 
\begin{equation}
\big[I_{\cH} - M(z)^{-1}\big] \in \cB_{\infty}(\cH),\; z \in \Omega\backslash\cD_1, 
\quad [I_{\cH} - D_0(z_1)] \in \cB_{\infty}(\cH), \; z_1 \in \cD_1.  
\end{equation} 
\end{theorem}

Assume Hypothesis \ref{h4.3}, let $z_0\in \Omega$ and suppose that $M(\cdot)$ is boundedly 
invertible on 
$D(z_0; \varepsilon_0) \backslash \{z_0\}$ for some $0 < \varepsilon_0$ sufficiently small (i.e., case $(ii)$ in Theorem \ref{t3.2} applies).
Then the function $M(\cdot)^{-1}$ is finitely meromorphic on $D(z_0; \varepsilon_0)$ and it follows from
the operator-valued version of the argument principle proved in \cite{GS71} (see also 
\cite[Theorem\ 4.4.1]{GL09}) that 
\begin{equation}\label{4.9}
\ind_{C(z_0; \varepsilon)}(M(\cdot)) \in \bbZ.
\end{equation}

\section{Abstract Perturbation Theory and Applications to 
Birman--Schwinger-Type Operators} \lb{s5}

In this section, following Kato \cite{Ka66}, Konno and Kuroda
\cite{KK66}, and Howland \cite{Ho70}, we first recall a class
of factorable non-self-adjoint perturbations of a given unperturbed
non-self-adjoint operator. We recall the treatment in \cite{GLMZ05} 
(in which $H_0$ is explicitly permitted to be non-self-adjoint, 
cf.\ Hypothesis \ref{h5.1}\,$(i)$ below) and refer to the latter for detailed proofs. 

The principal result of this section then consists of the index formulas in 
Theorem~\ref{t5.5},
which are variants of \cite[Theorem 4.5 and Theorem 5.5]{GHN15}.
We start with the following set of hypotheses.


\begin{hypothesis} \label{h5.1}
$(i)$ Suppose that $H_0\colon\dom(H_0)\to\cH$,
$\dom(H_0)\subseteq\cH$, is a densely defined, closed, linear operator
in $\cH$ with nonempty resolvent set,
\begin{equation}\label{5.1j}
\rho(H_0)\neq\emptyset, 
\end{equation}
$V_1 \colon \dom(V_1)\to\cK$, $\dom(V_1)\subseteq\cH$, a densely defined,
closed, linear operator from $\cH$ to $\cK$, and
$V_2 \colon \dom(V_2)\to\cK$, $\dom(V_2)\subseteq\cH$, a densely
defined, closed, linear operator from $\cH$ to $\cK$ such that
\begin{equation}\label{5.2j}
\dom(V_1)\supseteq\dom(H_0), \quad \dom(V_2)\supseteq\dom(H_0^*).
\end{equation}
In the following we denote
\begin{equation}\label{5.3j}
R_0(z)=(H_0-zI_{\cH})^{-1}, \quad z\in \rho(H_0). 
\end{equation}
$(ii)$ For some (and hence for all) $z\in\rho(H_0)$, the operator
$- V_1R_0(z)V_2^*$, defined on $\dom(V_2^*)$, has a
bounded extension in $\cK$, denoted by $K(z)$,
\begin{equation}\label{5.4j}
K(z)=-\ol{V_1 R_0(z) V_2^*} \in\cB(\cK). 
\end{equation}
$(iii)$ $1\in\rho(K(\zeta_0))$ for some $\zeta_0\in \rho(H_0)$.
\end{hypothesis}

Next, following Kato \cite{Ka66}, one introduces
\begin{equation}\label{5.5j}
R(z)=R_0(z)-\ol{R_0(z) V_2^*}[I_{\cK}-K(z)]^{-1} V_1 R_0(z)
\end{equation}
for $z\in\{\zeta\in\rho(H_0)\,|\, 1\in\rho(K(\zeta))\}$.
\begin{theorem} [\cite{Ka66}] \label{t5.2}
Assume Hypothesis \ref{h5.1} and 
$z\in\{\zeta\in\rho(H_0)\,|\, 1\in\rho(K(\zeta))\}$. Then, $R(z)$
in \eqref{5.5j} defines a densely defined, closed, linear
operator $H$ in $\cH$ by
\begin{equation}\label{5.6j}
R(z)=(H-zI_{\cH})^{-1}. 
\end{equation}
Moreover,
\begin{equation}\label{5.7j}
V_1 R(z),  V_2 R(z)^* \in \cB(\cH,\cK) 
\end{equation}
and
\begin{align}
\begin{split} 
R(z)&=R_0(z)-\ol{R(z) V_2^*} V_1 R_0(z)  \\
&=R_0(z)-\ol{R_0(z) V_2^*} V_1 R(z).   \lb{5.9j} 
\end{split} 
\end{align}
Finally, $H$ is an extension of the operator
\begin{equation}\label{5.10j}
(H_0 + V_2^* V_1)\upharpoonright \bigl(\dom(H_0)\cap\dom(V_2^* V_1)\bigr),
\end{equation}
where the set $\dom(H_0)\cap\dom(V_2^* V_1)$ may consist of $\{0\}$ only.
\end{theorem}

Similarly, using the symmetry between $H_0$ and $H$ inherent in Kato's formalism 
(cf.\ \cite[Sects.~2, 3]{GLMZ05}) one also derives
\begin{equation}
\ol{V_1 R(z) V_2^*} \in \cB(\cK), \quad z \in \rho(H),
\end{equation}
and
\begin{equation}
I_{\cK} - \ol{V_1 R(z) V_2^*} = [I_{\cK} - K(z)]^{-1}, \quad 
z \in \{\zeta \in \rho(H_0) \, | \, 1 \in \rho(K(\zeta))\}.     \lb{5.10f}
\end{equation}

For our purposes the following lemma is useful. 

\begin{lemma}\label{l5.3}
Assume Hypothesis \ref{h5.1} and let $z_1,z_2\in\rho(H_0)$. Then
\begin{equation}\label{5.11j}
 K(z_1)=K(z_2)+(z_2-z_1)V_1R_0(z_1)\overline{R_0(z_2)V_2^*}
\end{equation}
and if, in addition, $z_1,z_2\in\rho(H)$ then
\begin{equation}\label{5.12j}
 [I_\cK-K(z_1)]^{-1} = [I_\cK-K(z_2)]^{-1} + (z_2-z_1)V_1R(z_1)\overline{R(z_2)V_2^*}.
\end{equation}
\end{lemma}
\begin{proof}
Formula \eqref{5.11j} follows from \eqref{5.4j} and the resolvent equation for $R_0(z)$, 
$z \in \rho(H_0)$; similarly, formula \eqref{5.12j} is clear from \eqref{5.10f} and the resolvent 
equation for $R(z)$, $z \in \rho(H)$.
\end{proof}

Note also that \eqref{5.11j} yields the useful formula
\begin{equation}\label{5.11ju}
 K'(z)=-V_1R_0(z)\overline{R_0(z)V_2^*},\qquad z\in\rho(H_0).
\end{equation}

The next result represents an abstract version
of (a variant of) the Birman--Schwinger principle due to Birman \cite{Bi61}
and Schwinger \cite{Sc61} (cf.\ also \cite{BS91}, \cite{GH87},
\cite{Kl82}, \cite{KS80}, \cite{Ne83}, \cite{Ra80}, \cite{Se74}, \cite[Ch.\ III]{Si71}, and
\cite{Si77a}). It is due to Konno and Kuroda \cite{KK66} in the case where $H_0$ is self-adjoint. 
For the general case see \cite{GLMZ05}. 

\begin{theorem}[\cite{KK66}] \lb{t5.4}
Assume Hypothesis \ref{h5.1} and let $z_0\in\rho(H_0)$. Then,
\begin{equation}\label{sigmapj}
 z_0\in\sigma_p(H)\,\text{ if and only if }\, 1\in\sigma_p(K(z_0)),
\end{equation}
and
\begin{equation}\label{rhoj}
 z_0\in\rho(H)\,\text{ if and only if }\, 1\in\rho(K(z_0)).
\end{equation}

More precisely, if in \eqref{sigmapj} one has $Hf=z_0f$ for some $f\in\dom(H)$, $f\not=0$, then 
\begin{equation}\label{5.16j}
0\not= g = [I_{\cK}-K(z_1)]^{-1} V_1 R_0(z_1)f = (z_0-z_1)^{-1} V_1 f,
\end{equation}
where $z_1\in\{\zeta\in\rho(H_0)\,|\,
1\in\rho(K(\zeta))\}$, $z_1\neq z_0$, satisfies $K(z_0)g=g$,
and conversely, if in \eqref{sigmapj} one has $K(z_0)g=g$ for some $g\in\cK$, $g\not=0$, then 
\begin{equation}\label{5.17j}
0\neq f=-\ol{R_0(z_0) V_2^*}g\in\dom(H)
\end{equation}
satisfies $H f=z_0 f$.
\end{theorem}

If, in addition to Hypothesis~\ref{h5.1}, it is assumed that $I_\cK - K(z)$ is a Fredholm operator for all $z\in\rho(H_0)$, then by \cite[Theorem 2.7]{GHN15} (see also \cite[Theorem 3.2]{GLMZ05})
the geometric multiplicity of an eigenvalue $z_0$ of $H$ coincides with the geometric multiplicity of the eigenvalue $1$ of $K(z_0)$ and is finite, 
\begin{equation}\label{5.18j}
\begin{split}
m_g(z_0;H) &= \dim(\ker(H-z_0I_{\cH}))\\
&=\dim(\ker(I_{\cK}-K(z_0))) = m_g(1;K(z_0)) <\infty.
\end{split}
\end{equation}

The next theorem is the main result in this section. Item $(i)$ is a slight extension (cf.\ \cite{GHN15}) 
of a multiplicity result due to Latushkin and Sukhtyaev \cite{LS10}, and item $(ii)$ resembles an 
analog of the Weinstein--Aronszajn-type formula (cf., e.g., \cite{AB70}, \cite{Ho70}, 
\cite[Sect.\ IV.6]{Ka80}, \cite{Ku61}, \cite[Sect.\ 9.3]{WS72}) in the case where $H$ and $H_0$ have common discrete eigenvalues. 

\begin{theorem}\label{t5.5}
Assume Hypothesis~\ref{h5.1}. Then the following assertions $(i)$--$(iv)$ hold:\\[1mm] 
$(i)$ If $z_0 \in \rho(H_0) \cap \sigma_d(H)$, then 
 the index formula
\begin{equation}\label{indi0}
\ind_{C(z_0; \varepsilon)}(I_\cK-K(\cdot))=m_a(z_0;H)
\end{equation}
holds for $\varepsilon>0$ sufficiently small. Furthermore, 
$z_0$ is a zero of finite-type of the function $I_{\cK} - K(\cdot)$, and hence
\begin{equation}\label{indinu}
\nu(z_0; I_{\cK} - K(\cdot))=m_a(z_0; I_{\cK} - K(\cdot))=   \ind_{C(z_0; \varepsilon)}(I_\cK-K(\cdot)).   
\end{equation} 
$(ii)$  If  $z_0 \in \sigma_d(H_0) \cap \sigma_d(H)$,
then the index formula
\begin{equation}\label{indi}
\ind_{C(z_0; \varepsilon)}(I_\cK- K(\cdot))=m_a(z_0;H)-m_a(z_0;H_0)
\end{equation}
holds for $\varepsilon>0$ sufficiently small. \\[1mm]  
$(iii)$ Assume in addition that $K(z) \in \cB_{\infty}(\cK)$ for all $z \in \rho(H_0)$ and either that 
$\rho(H_0)$ is connected, or else, that Hypothesis \ref{h5.1}\,$(iii)$, that is, $1 \in \rho(K(\zeta))$, 
holds for some $\zeta \in \bbC$ lying in each of the connected components of $\rho(H_0)$. If 
$z_0 \in \sigma_d(H_0)$, then $z_0 \in (\sigma_d(H) \cup \rho(H))$ and hence the index formula 
\eqref{indi} holds. \\[1mm] 
$(iv)$ Assume in addition that $K(z) \in \cB_{\infty}(\cK)$ for all $z \in \rho(H_0)$ and suppose 
that $D(z_0;\varepsilon_0) \backslash \{z_0\}\cap\sigma(H)=\emptyset$ 
for $0 < \varepsilon_0$ sufficiently small. If $z_0 \in \sigma_d(H_0)$, then $z_0 \in (\sigma_d(H) \cup \rho(H))$
and hence the index formula \eqref{indi} holds. 
\end{theorem}
\begin{proof}
Observe first that by the assumptions in $(i)$ and $(ii)$ there exists $\varepsilon_0>0$ such that 
the punctured disc $D(z_0;\varepsilon_0) \backslash \{z_0\}$ is contained in $\rho(H)\cap\rho(H_0)$. Fix a point $z_2\in\rho(H)\cap\rho(H_0)$ and recall from Lemma~\ref{l5.3}\,$(i)$  
that 
\begin{equation}\label{useful}
 K(z)=K(z_2)+(z_2-z)V_1(H_0-zI_{\cH})^{-1}\overline{R_0(z_2)V_2^*}
\end{equation}
holds for all $z\in D(z_0;\varepsilon_0)$ if $z_0\in\rho(H_0)$ and for all $z\in D(z_0;\varepsilon_0) \backslash \{z_0\}$ if $z_0\in\sigma_d(H_0)$. Therefore, 
since $(H_0-z I_{\cH})^{-1}$ is analytic on $D(z_0;\varepsilon_0)$ if $z_0\in\rho(H_0)$ and finitely meromorphic if $z_0\in\sigma_d(H_0)$ (see, e.g., 
\cite[Chapter 1, $\S$2. Theorem 2.1 and (2.3)]{GK69} or \cite{Ka80}) it follows from
\eqref{useful} and $V_1(H_0-z I_{\cH})^{-1}\in\cB(\cH,\cK)$ that the same is true for the functions $K(\cdot)$ and $I_\cK-K(\cdot)$. The same argument using the resolvent of $H$ and formula \eqref{5.12j} in Lemma~\ref{l5.3} shows that the function $[I_\cK - K(\cdot)]^{-1}$ is analytic on the punctured disc $D(z_0;\varepsilon_0) \backslash \{z_0\}$ and 
finitely meromorphic on  $D(z_0; \varepsilon_0)$. Hence the index of $I_\cK-K(\cdot)$ with respect to the counterclockwise oriented circle $C(z_0,\varepsilon)$,
$0<\varepsilon<\varepsilon_0$, is well-defined and we compute with the help of \eqref{5.11ju}, the cyclicity of the trace, and \eqref{5.5j}
 \begin{align} 
 \ind_{C(z_0; \varepsilon)} \big(I_\cK - K(\cdot)\big) 
& = \tr_\cK\bigg(\f{1}{2 \pi i} \ointctrclockwise_{C(z_0; \varepsilon)} d\zeta \, 
[I_\cK - K(\zeta)]^{-1} \bigl(-K^\prime(\zeta)\bigr) \bigg)     \no \\
& = \f{1}{2 \pi i} \tr_\cK\bigg(\ointctrclockwise_{C(z_0; \varepsilon)} d\zeta \, 
[I_\cK - K(\zeta)]^{-1} V_1 R_0(\zeta)\overline{R_0(\zeta)V_2^*} \bigg)    \no \\
& = \f{1}{2 \pi i} \tr_{\cH}\bigg(\ointctrclockwise_{C(z_0; \varepsilon)} d\zeta \, 
\overline{R_0(\zeta)V_2^*} [I_\cK - K(\zeta)]^{-1} V_1 R_0(\zeta) \bigg)     \no \\
& = \tr_{\cH}\bigg(\f{-1}{2 \pi i} \ointctrclockwise_{C(z_0; \varepsilon)} d\zeta \, 
\bigl(R(\zeta)-R_0(\zeta)\bigr) \bigg)      \no \\
&=m_a(z_0; H) - m_a(z_0; H_0).     \lb{5.24} 
\end{align}
Here the third equality in \eqref{5.24} follows in analogy to \eqref{2.5j} 
(cf.\ \cite[Proposition~4.2.2]{GL09}), and the last equality holds if 
$z_0\in\sigma_d(H_0)$. This proves the index formula \eqref{indi}. 
Clearly, if $z_0\in\rho(H_0)$ then 
\begin{equation}
\tr_{\cH}\bigg(\f{-1}{2 \pi i} \ointctrclockwise_{C(z_0; \varepsilon)} d\zeta \, R_0(\zeta) \bigg) =0
\end{equation}
and hence the term $m_a(z_0; H_0)$ is absent in the above computation; this implies the index formula \eqref{indi0}. 

Next, we show that $z_0\in\rho(H_0)\cap\sigma_d(H)$ is a zero of finite-type of $I_{\cK} - K(\cdot)$; the first equality in \eqref{indinu} then follows
from Theorem~\ref{t3.4a} and the second equality is clear by \eqref{coincide}. In order to see that $z_0$ is a zero of finite-type recall that
$I_\cK-K(z)$ is boundedly invertible for all $z\in D(z_0;\varepsilon_0) \backslash \{z_0\}$ and that $z_0\in\sigma_d(H)$ and Theorem~\ref{t5.4} imply 
$\dim(\ker(I_{\cK}-K(z_0)))<\infty$. Moreover, the function $[I_\cK - K(\cdot)]^{-1}$ is finitely meromorphic on  $D(z_0; \varepsilon_0)$ 
and it follows from the particular form of the Laurent series of $(H-zI_{\cH})^{-1}$ in a neighborhood of $z_0\in\sigma_d(H)$
(see, e.g., \cite[Chapter 1, $\S$2. Theorem 2.1 and (2.3)]{GK69}) and 
Lemma~\ref{l5.3}\,$(ii)$ that the zero order coefficient of the Laurent series
of $[I_\cK - K(\cdot)]^{-1}$ in a neighborhood of $z_0$ is a Fredholm operator, that is, Hypothesis~\ref{h4.3} is satisfied for $[I_\cK - K(\cdot)]^{-1}$. 
Hence, Theorem~\ref{t3.2} applies to the function $[I_{\cK}-K(\cdot)]^{-1}$ and from \eqref{jussi4fred} we obtain that $I_{\cK}-K(z_0)$ is a Fredholm operator.
Summing up we have shown that $z_0$ is a zero of finite-type of $I_{\cK} - K(\cdot)$.

We turn to a discussion of item $(iii)$. If $z_0 \in \rho(H)$, no proof is required and the index formula (5.23) takes the form
\begin{equation}\label{indiA}
\ind_{C(z_0; \varepsilon)}(I_\cK- K(\cdot)) = - m_a(z_0;H_0).
\end{equation} 
So we focus 
on $z_0 \in \sigma(H)$.\ From the outset it is clear that for $0 < \varepsilon_0$ sufficiently 
small, $\ol{R_0(z) V_2^*}$, $V_1 R_0(z)$, and $K(z)$ are analytic on 
$D(z_0; \varepsilon_0) \backslash \{z_0\}$, and $K(z)$ is finitely meromorphic on $D(z_0;\varepsilon_0)$.\ 
In particular, $K(z)$, $z \in D(z_0; \varepsilon_0) \backslash \{z_0\}$, is of the form,
\begin{equation}\label{5.26f}
K(z) = \sum_{k=-N_0}^{\infty} (z-z_0)^k K_k(z_0), \quad 
0 < |z - z_0| < \varepsilon_0,   
\end{equation}
for some $N_0 \in \bbN$, with $K_k(z_0) \in \cF(\cK)$, $-N_0 \leq k \leq -1$, 
$K_k(z_0) \in \cB(\cK)$, $k \in \bbN_0$. Hence, 
\begin{equation}
\bigg[\sum_{k=0}^{\infty} (z-z_0)^k K_k(z_0)\bigg] \in \cB_{\infty}(\cK), \quad 
0 < |z - z_0| < \varepsilon_0,   
\end{equation} 
implying that the norm limit,  
\begin{equation}
K_0(z_0) = \lim_{z \to z_0} 
\bigg[\sum_{k=0}^{\infty} (z-z_0)^k K_k(z_0)\bigg] \in \cB_{\infty}(\cK),  
\end{equation} 
exists and is compact. In particular, this implies 
\begin{equation}
[I_{\cK} - K_0(z_0)] \in \Phi(\cK).     \lb{5.29f}
\end{equation}
If $\rho(H_0)$ is connected then Hypothesis~\ref{h5.1}\,$(iii)$ and $K(z)\in\cB_{\infty}(\cK)$, 
$z\in\rho(H_0)$, imply that $I_{\cK} - K(z)$ is boundedly invertible for some 
$z\in D(z_0; \varepsilon_0) \backslash \{z_0\}$.
If $\rho(H_0)$ is not connected then the assumption  $1 \in \rho(K(\zeta))$ for some 
$\zeta \in \bbC$ in each of the connected components of $\rho(H_0)$ implies in the same way that
$I_{\cK} - K(z)$ is boundedly invertible for some $z\in D(z_0; \varepsilon_0) \backslash \{z_0\}$.
Consequently, Theorems \ref{t2.4}\,$(ii)$, respectively, \ref{t3.2}\,$(ii)$, apply, and hence $[I_{\cK} - K(z)]^{-1}$ is analytic 
on $D(z_0; \varepsilon_0) \backslash \{z_0\}$, respectively, finitely meromorphic on $D(z_0;\varepsilon_0)$ (possibly, upon further diminishing $\varepsilon_0 > 0$). By \eqref{5.5j}, 
then also $R(z)$ is analytic on $D(z_0; \varepsilon_0) \backslash \{z_0\}$ and finitely meromorphic 
on $D(z_0;\varepsilon_0)$, implying $z_0 \in \sigma_d(H)$. 

Finally, we briefly turn to item $(iv)$ again assuming $z_0 \in \sigma(H)$ without loss of generality. 
By \eqref{5.10f}, the condition $(D(z_0; \varepsilon_0)\setminus\{z_0\}) \cap \sigma(H) = \emptyset$ guarantees the bounded invertibility of $I_{\cK} - K(z)$ for 
$z \in D(z_0; \varepsilon_0) \backslash \{z_0\}$ and one can now basically follow the proof of item 
$(iii)$; we omit the details. 
\end{proof}

\begin{remark} \lb{r5.6}
In connection with Theorem \ref{t5.5}\,$(iii)$, one notes that since $\rho(H_0) \subset \bbC$ is 
open, its connected components are open and at most countable (see, e.g., 
\cite[Theorem~II.2.9]{Co78}). In particular, in the important special case where 
$\sigma(H_0) \subseteq \bbR$, there are at most two components and in quantum mechanical applications associated with short-range potential coefficients one frequently encounters that 
\begin{equation} 
\lim_{y \to \pm \infty} \|K(iy)\|_{\cB(\cK)} = 0,  
\end{equation}
and hence the condition $1 \in \rho(K(\zeta))$ 
is obviously satisfied for $\zeta = iy$ with $0 < |y|$ sufficiently large. 

In this context we note that condition \eqref{5.29f}, that is, $[I_{\cK} - K_0(z_0)] \in \Phi(\cK)$, was inadvertently omitted in \cite[Theorem~5.5]{GHN15} and hence needs to be added to its 
hypotheses. \hfill $\diamond$
\end{remark}

\section{An Index Formula for the Weyl--Titchmarsh Function Associated to Closed 
Extensions of Dual Pairs}  \lb{s6} 

In this section we derive the index associated with the Weyl--Titchmarsh function associated 
to closed extensions of dual pairs of operators.

Let $\cK$ be a separable, complex Hilbert space with scalar product $(\cdot,\cdot)_\cK$, 
and let $A$ and $B$ be densely defined, closed, linear operators in $\cK$ such that
\begin{equation}\label{6.1j}
 (Bf,g)_\cK=(f,Ag)_\cK,\quad f\in\dom(B), \; g\in\dom(A).
\end{equation}
A pair of operators $\{A,B\}$ that satisfies \eqref{6.1j} is called a {\it dual pair}. It follows immediately from \eqref{6.1j} that
\begin{equation}\label{6.1jjj}
 A\subset B^* \, \text{ and } \, B\subset A^*.
\end{equation}

We recall the notion of a boundary triple for a dual pair from \cite{MM02} (see also \cite{MM03}, \cite{MMH12}).

\begin{definition}\label{defdef}
Let $\{A,B\}$ be a dual pair of operators in $\cK$. A triple $\{\cH,\Gamma^B,\Gamma^A\}$, 
where $\cH=\cH_0\oplus\cH_1$ is  a 
Hilbert space and 
\begin{equation}\label{6.3j}
\Gamma^B=\big(\Gamma_0^B,\Gamma_1^B\big)^\top:\dom (B^*)\rightarrow \cH_0\oplus\cH_1
\end{equation}
and 
\begin{equation}\label{6.3jj}
\Gamma^A=\big(\Gamma_0^A,\Gamma_1^A\big)^\top:\dom (A^*)\rightarrow \cH_1\oplus\cH_0,
\end{equation}
are linear mappings, is called a {\em boundary triple} for the dual pair $\{A,B\}$ if the following 
items $(i)$--$(ii)$ hold: \\[1mm] 
$(i)$ For all $f\in\dom(B^*)$ and $g\in\dom(A^*)$, the following abstract Green's identity 
\hspace*{4mm} holds, 
 \begin{equation}\label{6.4j}
  (B^*f,g)_\cK-(f,A^*g)_\cK= \big(\Gamma_1^B f,\Gamma_0^A g\big)_{\cH_1} 
  - (\Gamma_0^B f,\Gamma_1^A g)_{\cH_0}.
 \end{equation}
$(ii)$ The mappings $\Gamma^B$ and $\Gamma^A$  in \eqref{6.3j} and \eqref{6.3jj} are both onto.
\end{definition}

Next, assume that $\{A,B\}$ is a dual pair of operators in $\cK$ and that $\{\cH,\Gamma^B,\Gamma^A\}$, $\cH=\cH_0\oplus\cH_1$, is a boundary triple for $\{A,B\}$.
Then one has 
\begin{equation}\label{6.5j}
 A=B^*\upharpoonright\ker\big(\Gamma^B\big) \, \text{ and } \, 
 B=A^*\upharpoonright\ker\big(\Gamma^A\big),
\end{equation}
and the mappings in \eqref{6.3j} and \eqref{6.3jj} are continuous when $\dom (B^*)$ and $\dom (A^*)$ are equipped with the graph norm. Moreover,
the closed operators
\begin{equation}\label{6.6j}
\begin{split}
 A_0=B^*\upharpoonright\ker\big(\Gamma_0^B\big) & \, \text{ and } \, 
 A_1=B^*\upharpoonright\ker\big(\Gamma_1^B\big), \\
 B_0=A^*\upharpoonright\ker\big(\Gamma_0^A\big) & \, \text{ and } \, 
 B_1=A^*\upharpoonright\ker\big(\Gamma_1^A\big),
 \end{split}
 \end{equation}
satisfy $B_0=A_0^*$ and $B_1=A_1^*$, and 
\begin{equation}\label{6.6jjj}
A\subset A_0, \; A_1\subset B^* \, \text{ and } \, B\subset B_0, \; B_1\subset A^*.
\end{equation}
More generally,
with the help of a boundary triple for the dual pair $\{A,B\}$ one can describe all closed extensions $A_\Theta$ of $A$ that are restrictions of $B^*$,
that is,
\begin{equation}\label{6.8j}
 A\subset A_\Theta\subset B^*
\end{equation}
with the help of closed linear subspaces $\Theta$ in $\cH_0\times\cH_1$. 
We refer the reader to \cite{MM02} and \cite{MM03} for more details and concentrate
on the special case of extensions of $A$ of the form
\begin{equation}\label{6.9j}
 A_\Theta=B^*\upharpoonright\ker\bigl(\Gamma_1^B-\Theta\Gamma_0^B\bigr),
\end{equation}
where we assume that $\Theta\in\cB(\cH_0,\cH_1)$ is a bounded operator from $\cH_0$ 
into $\cH_1$.

In order to state our main result in this context some more definitions are necessary. First, we recall 
the notion of {\it $\gamma$-field} and {\it Weyl--Titchmarsh function} associated to a boundary triple for a dual pair treated in \cite{MM02} and \cite{MM03}.
Suppose that $\rho(A_0)\not=\emptyset$, $\rho(B_0)\not=\emptyset$, and observe that the direct sum decompositions
\begin{equation}\label{6.10j}
 \dom (B^*)=\dom (A_0)\,\dot+\,\ker (B^*- z I_{\cK}),\quad z \in\rho(A_0),
 \end{equation}
 and
 \begin{equation}\label{6.11j}
 \dom (A^*)=\dom (B_0)\,\dot+\,\ker (A^*-z' I_{\cK}),\quad z' \in\rho(B_0),
\end{equation}
hold. Since 
\begin{equation}\label{6.12j}
\dom(A_0)=\ker\big(\Gamma_0^B\big), \quad \dom(B_0)=\ker\big(\Gamma_0^A\big), 
\end{equation}
it follows from \eqref{6.10j} that the mapping $\Gamma_0^B$
is invertible on $\ker (B^*- z I_{\cK})$, and it follows from \eqref{6.11j} that the mapping $\Gamma_0^A$
is invertible on $\ker (A^*- z' I_{\cK})$.

\begin{definition}
Let $\{A,B\}$ be a dual pair of operators in $\cK$ and let $\{\cH,\Gamma^B,\Gamma^A\}$ be a boundary triple.
The {\em $\gamma$-fields} $\gamma(\cdot)$ and $\gamma_*(\cdot)$ associated to  $\{\cH,\Gamma^B,\Gamma^A\}$ are defined by
\begin{equation}\label{6.13j}
 \gamma(z)=\bigl(\Gamma_0^B\upharpoonright\ker(B^*-z I_{\cK})\bigr)^{-1},\quad z\in\rho(A_0),
 \end{equation}
 and
 \begin{equation}\label{6.14j}
 \gamma_*(z')=\bigl(\Gamma_0^A\upharpoonright\ker(A^*- z' I_{\cK})\bigr)^{-1},\quad z'\in\rho(B_0),
\end{equation}
respectively. The {\em Weyl--Titchmarsh function} $M(\cdot)$ associated to  $\{\cH,\Gamma^B,\Gamma^A\}$ is defined by
\begin{equation}\label{6.15j}
 M(z)=\Gamma_1^B \bigl(\Gamma_0^B\upharpoonright\ker(B^*-z I_{\cK})\bigr)^{-1},\quad\ z\in\rho(A_0).
\end{equation}
\end{definition}

It is important to note that the $\gamma$-field satisfies
\begin{equation}\label{6.16j}
 \gamma(z_1)=\bigl(I_\cK+(z_1 - z_2)(A_0 - z_1 I_{\cK})^{-1}\bigr)\gamma(z_2),
 \quad z_j \in \rho(A_0), \; j=1,2.
\end{equation}
Moreover, the values $M(z)$ of the Weyl--Titchmarsh function are bounded operators from $\cH_0$ to $\cH_1$,
\begin{equation}\label{6.17j}
 M(z)\in\cB(\cH_0,\cH_1), \quad z \in \rho(A_0), 
\end{equation}
and the Weyl--Titchmarsh function and the $\gamma$-fields are related via
\begin{equation}\label{6.18j}
 M(z_1)-M(z_2)=(z_1 - z_2)\gamma_*(\overline{z_2})^*\gamma(z_1), 
 \quad  z_j \in \rho(A_0), \; j=1,2.
\end{equation}

We shall assume from now on that $\{A,B\}$ is a dual pair and $\{\cH,\Gamma^B,\Gamma^A\}$ 
is a boundary triple with the additional property $\cH_0=\cH_1$, which can be viewed as a  
non-symmetric analog of the case of equal deficiency indices of an underlying symmetric operator. 
Consider a closed extension $A_\Theta$ of $A$ as in \eqref{6.9j} with $\Theta\in\cB(\cH_0)$, and assume that $z \in \rho(A_0)$.
Then by \cite[Proposition 5.2]{MM02} one has 
\begin{equation}\label{6.19j}
 z \in \sigma_p(A_\Theta) \, \text{ if and only if } \, 0\in\sigma_p(\Theta-M(z)), 
\end{equation}
and
\begin{equation}\label{6.19j2}
 z \in \rho(A_\Theta) \, \text{ if and only if } \, 0\in\rho(\Theta-M(z)). 
\end{equation}
Moreover, for all $z \in \rho(A_0) \cap \rho(A_\Theta)$, the following Krein-type resolvent 
formula holds, 
\begin{equation}\label{6.20j}
(A_{\Theta} - z I_{\cK})^{-1} = (A_0 - z I_{\cK})^{-1} +\gamma(z)  
[\Theta - M(z)]^{-1} \gamma_*(\overline{z})^*.
\end{equation}

The next lemma will be useful in the proof of our main result Theorem~\ref{t6.4} below (cf.\  \cite[Corollary 4.9]{MM02}).
For the convenience of the reader we provide a simple direct proof in the present situation.

\begin{lemma}\label{lem6.3}
Let $\{A,B\}$ be a dual pair of operators in $\cK$, let $\{\cH,\Gamma^B,\Gamma^A\}$ be a boundary triple 
with $ A_0=B^*\upharpoonright\ker(\Gamma_0^B)$ and Weyl--Titchmarsh function $M(\cdot)$, and assume that $\cH_0=\cH_1$. 
Suppose that $\Theta\in\cB(\cH_0)$ and let $A_\Theta$ be defined as in \eqref{6.9j}. Then $\{\cH,\Gamma^{B,\Theta},\Gamma^{A,\Theta}\}$, where
\begin{equation}\label{6.21jjj}
\Gamma^{B,\Theta}=\begin{pmatrix} \Gamma_0^{B,\Theta} \\[1mm] \Gamma_1^{B,\Theta}\end{pmatrix},\quad \Gamma_0^{B,\Theta}=\Gamma_1^B-\Theta\Gamma_0^B,\quad \Gamma_1^{B,\Theta}=-\Gamma_0^B,
\end{equation}
and
\begin{equation}\label{6.21ajjj}
\Gamma^{A,\Theta}=\begin{pmatrix} \Gamma_0^{A,\Theta} \\[1mm] \Gamma_1^{A,\Theta}\end{pmatrix},\quad
\Gamma_0^{A,\Theta}=\Gamma_1^A-\Theta^*\Gamma_0^A,\quad \Gamma_1^{A,\Theta}=-\Gamma_0^A,
\end{equation}
is a boundary triple for the dual pair $\{A,B\}$ with 
$A_\Theta=B^*\upharpoonright\ker\big(\Gamma_0^{B,\Theta}\big)$.~The corresponding 
Weyl--Titchmarsh function $M_\Theta(\cdot)$ is given by
\begin{equation}\label{mt}
 M_\Theta(z)=\bigl(\Theta-M(z)\bigr)^{-1}, \quad z \in \rho(A_\Theta)\cap\rho(A_0).
\end{equation}
\end{lemma}
\begin{proof}
Let $f\in\dom(B^*)$ and $g\in\dom(A^*)$. Then it follows with the help of the abstract Green's identity \eqref{6.4j} for the boundary triple
$\{\cH,\Gamma^B,\Gamma^A\}$ that
\begin{equation}\label{btcomp}
\begin{split}
& \bigl(\Gamma_1^{B,\Theta}f,\Gamma_0^{A,\Theta}g\bigr)_{\cH_0}- \bigl(\Gamma_0^{B,\Theta}f,\Gamma_1^{A,\Theta}g\bigr)_{\cH_0}\\
& \quad =\bigl(-\Gamma_0^Bf,\Gamma_1^Ag-\Theta^*\Gamma_0^Ag\bigr)_{\cH_0}-\bigl(\Gamma_1^Bf-\Theta \Gamma_0^Bf,-\Gamma_0^Ag\bigr)_{\cH_0}\\
& \quad = \big(\Gamma_1^B f,\Gamma_0^A g\big)_{\cH_0} 
- \big(\Gamma_0^B f,\Gamma_1^A g\big)_{\cH_0}\\
& \quad = (B^*f,g)_\cK-(f,A^*g)_\cK,
\end{split}
\end{equation}
and hence the triple $\{\cH,\Gamma^{B,\Theta},\Gamma^{A,\Theta}\}$ satisfies the abstract Green's identity in Definition~\ref{defdef}\,$(i)$.
Moreover, as
\begin{equation}\label{jaok}
 \begin{pmatrix}
 \Gamma_0^{B,\Theta}\\[1mm] \Gamma_1^{B,\Theta} 
 \end{pmatrix}
=W_B^\Theta
\begin{pmatrix}
 \Gamma_0^{B}\\[1mm] \Gamma_1^{B} 
 \end{pmatrix},\quad W_B^\Theta=\begin{pmatrix} 
  -\Theta & I_{\cH_0} \\ -I_{\cH_0} & 0
 \end{pmatrix},
\end{equation}
and
\begin{equation}\label{jaok2}
 \begin{pmatrix}
 \Gamma_0^{A,\Theta}\\[1mm] \Gamma_1^{A,\Theta} 
 \end{pmatrix}
=W_A^{\Theta^*}
\begin{pmatrix}
 \Gamma_0^{A}\\[1mm] \Gamma_1^{A} 
 \end{pmatrix},\quad W_A^{\Theta^*}=\begin{pmatrix} 
  -\Theta^* & I_{\cH_0} \\ -I_{\cH_0} & 0
 \end{pmatrix},
\end{equation}
and the $2\times 2$ block operator matrices $W_B^\Theta$ and $W_A^{\Theta^*}$ in \eqref{jaok} and \eqref{jaok2} are boundedly invertible, it follows that both mappings
\begin{equation}\label{6.3aj}
\Gamma^{B,\Theta}=\big(\Gamma_0^{B,\Theta},\Gamma_1^{B,\Theta}\big)^\top:\dom (B^*)\rightarrow \cH_0\oplus\cH_0
\end{equation}
and 
\begin{equation}\label{6.3ajj}
\Gamma^{A,\Theta}=\big(\Gamma_0^{A,\Theta},\Gamma_1^{A,\Theta}\big)^\top:\dom (A^*)\rightarrow \cH_0\oplus\cH_0,
\end{equation}
are onto. Hence also condition $(ii)$ in Definition~\ref{defdef} holds for $\{\cH,\Gamma^{B,\Theta},\Gamma^{A,\Theta}\}$, and it follows that
$\{\cH,\Gamma^{B,\Theta},\Gamma^{A,\Theta}\}$ is a boundary triple for the dual pair $\{A,B\}$. By construction, one has (cf.\ \eqref{6.9j})
\begin{equation}
 B^*\upharpoonright\ker(\Gamma_0^{B,\Theta})=B^*\upharpoonright\ker\bigl(\Gamma_1^B-\Theta\Gamma_0^B\bigr)=A_\Theta. 
\end{equation}  

Next, it will be verified that the Weyl--Titchmarsh function $M_\Theta(\cdot)$ corresponding to the boundary triple $\{\cH,\Gamma^{B,\Theta},\Gamma^{A,\Theta}\}$
has the form \eqref{mt}. Assume that $f_z \in \ker(B^*- z I_{\cK})$ and that 
$z \in \rho(A_0)\cap\rho(A_\Theta)$. Since $M(\cdot)$ is the Weyl--Titchmarsh function
of the boundary triple $\{\cH,\Gamma^B,\Gamma^A\}$, one has 
$M(z)\Gamma_0^Bf_z = \Gamma_1^B f_z$, and 
hence it follows that
\begin{equation}\label{ayx}
\begin{split}
 [\Theta-M(z)]\Gamma_1^{B,\Theta}f_z 
 &=- [\Theta-M(z)]\Gamma_0^Bf_z \\
 &=-\Theta\Gamma_0^B f_z + \Gamma_1^Bf_z  \\
 &=\Gamma_0^{B,\Theta}f_z.
\end{split}
 \end{equation}
From the direct sum decomposition
\begin{equation}\label{ddd}
\begin{split}
 \dom (B^*)&=\dom(A_\Theta)\,\dot +\,\ker(B^*- z I_{\cK})\\
 &=\ker(\Gamma_0^{B,\Theta})\,\dot +\,\ker(B^*- z I_{\cK}), \quad z \in \rho(A_\Theta),
 \end{split}
\end{equation}
and the fact that $\Gamma_0^{B,\Theta}$ maps onto $\cH_0$ one then concludes together with \eqref{ayx} that $[\Theta-M(z)]$ maps onto $\cH_0$. Moreover, one has 
\begin{equation}\label{kern}
\ker(\Theta-M(z))=\{0\}.
\end{equation}
In fact, if $\Theta\varphi=M(z)\varphi$ for some $\varphi\in\cH_0$, then by \eqref{6.10j} there exists an element $f_z \in \ker(B^*- z I_{\cK})$ such that 
$\Gamma_0^B f_z = \varphi$. This leads to
\begin{equation}
 \Theta\Gamma_0^B f_z = \Theta\varphi=M(z)\varphi=M(z)\Gamma_0^B f_z   
 = \Gamma_1^B f_z, 
\end{equation}
and hence $f_z \in \dom(A_\Theta)\cap\ker(B^*- z I_{\cK})$. Therefore, 
$f_z\in\ker(A_\Theta- z I_{\cK})$, and as $z \in \rho(A_\Theta)$ by assumption, 
we conclude $f_z = 0$ and $\varphi=\Gamma_0^B f_z = 0$. This shows \eqref{kern}.
Now it follows from \eqref{ayx} that
\begin{equation}
[\Theta-M(z)]^{-1}\Gamma_0^{B,\Theta} f_z = \Gamma_1^{B,\Theta}f_z 
\end{equation}
for all $f_z \in \ker(B^*-z I_{\cK})$ and $z \in \rho(A_0)\cap\rho(A_\Theta)$. This finally implies that the 
Weyl--Titchmarsh function $M_\Theta(\cdot)$ has the form \eqref{mt}.
\end{proof}

The next theorem is the main result of this section. As in Lemma~\ref{lem6.3} we shall assume here that the boundary triple
$\{\cH,\Gamma^B,\Gamma^A\}$ has the additional property
$\cH_0=\cH_1$. 

\begin{theorem}\label{t6.4}
Let $\{A,B\}$ be a dual pair of operators in $\cK$, let $\{\cH,\Gamma^B,\Gamma^A\}$ be a boundary triple 
with $ A_0=B^*\upharpoonright\ker(\Gamma_0^B)$ and Weyl--Titchmarsh function $M(\cdot)$, and assume that $\cH_0=\cH_1$. Furthermore, let 
$\Theta\in\cB(\cH_0)$ be a bounded operator and consider the extension 
\begin{equation}
A_\Theta=B^*\upharpoonright\ker\bigl(\Gamma_1^B-\Theta\Gamma_0^B\bigr).
\end{equation}
Then the following assertions $(i)$ and $(ii)$ hold: \\
$(i)$ If $z_0 \in \rho(A_0) \cap \sigma_d(A_\Theta)$, then 
 the index formula
\begin{equation}\label{indi02}
\ind_{C(z_0; \varepsilon)}(\Theta-M(\cdot))=m_a(z_0;A_\Theta)
\end{equation}
holds for $\varepsilon>0$ sufficiently small. Furthermore, 
$z_0$ is a zero of finite-type of the function $\Theta - M(\cdot)$, and hence
\begin{equation}\label{indinu2}
\nu(z_0; \Theta - M(\cdot))=m_a(z_0; \Theta - M(\cdot))=   \ind_{C(z_0; \varepsilon)}(\Theta-M(\cdot)).   
\end{equation} 
$(ii)$ If  $z_0 \in \sigma_d(A_0) \cap \sigma_d(A_\Theta)$,
then the index formula
\begin{equation}\label{indi2}
\ind_{C(z_0; \varepsilon)}(\Theta-M(\cdot))=m_a(z_0;A_\Theta)-m_a(z_0;A_0)
\end{equation}
holds for $\varepsilon>0$ sufficiently small.
\end{theorem}
\begin{proof}
Choose $\varepsilon_0>0$ such that $D(z_0;\varepsilon_0) \backslash \{z_0\}\subset\rho(A_0)\cap\rho(A_\Theta)$.
 Then it follows from \eqref{6.16j} and \eqref{6.18j} that the Weyl--Titchmarsh function admits the representation
 \begin{align}\label{6.21j}
 \begin{split} 
  M(z_1) = M(z_2)+(z_1 - z_2)\gamma_*(\overline{z_2})^*\bigl(I_\cK+(z_1 - z_2)
  (A_0- z_1 I_{\cK})^{-1}\bigr)\gamma(z_2),&    \\ 
  z_j \in \rho(A_0), \; j=1,2.& 
  \end{split}
 \end{align}
If $z_0$ is a point in $\rho(A_0)$ then the resolvent $(A_0- z I_{\cK})^{-1}$
 is analytic on a disc $D(z_0; \varepsilon_0)$ with $\varepsilon_0>0$ sufficiently small, and if $z_0$ 
 is a
 discrete eigenvalue of $A_0$ the resolvent $(A_0- z I_{\cK})^{-1}$
 is analytic on a punctured disc $D(z_0; \varepsilon_0) \backslash \{z_0\}$ with $\varepsilon_0>0$ sufficiently small, and finitely meromorphic on the disc $D(z_0; \varepsilon_0)$.
 Hence one concludes from \eqref{6.21j} that in the case $z_0\in\rho(A_0)$ also the 
 Weyl--Titchmarsh function $M(\cdot)$ is analytic on the disc $D(z_0; \varepsilon_0)$,
 and in the case $z_0\in\sigma_d(A_0)$ the Weyl--Titchmarsh function $M(\cdot)$ is analytic on the
 punctured disc $D(z_0; \varepsilon_0) \backslash \{z_0\}$ and 
 finitely meromorphic on $D(z_0; \varepsilon_0)$. 
 It is clear that the same is also true for the function
 \begin{equation}\label{6.22j}
  \Theta-M(\cdot).
 \end{equation}
 Similarly, consider the boundary triple $\{\cH,\Gamma^{B,\Theta},\Gamma^{A,\Theta}\}$ in 
 Lemma~\ref{lem6.3} and the corresponding Weyl--Titchmarsh function 
 \begin{equation}\label{6.22jj}
 M_\Theta(z)= [\Theta-M(z)]^{-1},
 \end{equation}
 where the operators $M_\Theta(z)\in\cB(\cH_0)$
 are well-defined for all $z \in \rho(A_0)\cap\rho(A_\Theta)$.
 If $\gamma_\Theta(\cdot)$ and $\gamma_{*\Theta}(\cdot)$ denote the corresponding 
 $\gamma$-fields, then one has (cf.\ \eqref{6.21j})
 \begin{align}\label{6.21jj} 
& M_\Theta(z_1) = M_\Theta(z_2) + (z_1 - z_2)\gamma_{*\Theta}(\overline{z_2})^*
  \big[I_\cK+(z_1 - z_2)(A_\Theta - z_1 I_{\cK})^{-1}\big]\gamma_\Theta(z_2),    \no \\
& \hspace*{8cm} z_j \in \rho(A_\Theta), \; j=1,2. 
 \end{align}
Since  $z_0\in\sigma_d(A_\Theta)$ by assumption, it follows as above from the properties of the resolvent $(A_\Theta- z I_{\cK})^{-1}$ 
 and \eqref{6.21jj} that the function $M_\Theta(\cdot)$ in \eqref{6.22jj} is finitely meromorphic on the disc $D(z_0; \varepsilon_0)$. 
Furthermore, one obtains from 
\eqref{6.18j} that
\begin{equation}\label{6.23j}
 \frac{d}{dz} [\Theta-M(z)] = - \frac{d}{dz} M(z) = - \gamma_*(\overline{z})^*\gamma(z), 
 \quad z \in \rho(A_0),
\end{equation}
and hence one computes for $0 < \varepsilon<\varepsilon_0$  with the help of \eqref{6.23j} and \eqref{6.20j}, 
\begin{align}
 \ind_{C(z_0; \varepsilon)} \big(\Theta - M(\cdot)\big) 
& = \f{1}{2 \pi i} \tr_{\cH_0} \bigg(\ointctrclockwise_{C(z_0; \varepsilon)} d\zeta \, 
[\Theta - M(\zeta)]^{-1} (-M'(\zeta))  \bigg)    \no \\
&  = \f{- 1}{2 \pi i} \tr_{\cH_0}\bigg(\ointctrclockwise_{C(z_0; \varepsilon)} d\zeta \,  
[\Theta - M(\zeta)]^{-1} \gamma_*(\overline\zeta)^*\gamma(\zeta) \bigg) 
\no \\
&  = \f{- 1}{2 \pi i} \tr_\cK \bigg(\ointctrclockwise_{C(z_0; \varepsilon)} d\zeta \,  
\gamma(\zeta) \, [\Theta - M(\zeta)]^{-1} \gamma_*(\ol{\zeta})^* \bigg) 
\no \\ 
& = \f{- 1}{2 \pi i} \tr_{\cK} \bigg(\ointctrclockwise_{C(z_0; \varepsilon)} d\zeta \, 
\big[(A_{\Theta} - \zeta I)^{-1} - (A_0 - \zeta I)^{-1}\big] \bigg)    \no \\
&  = m_a(z_0; A_{\Theta}) - m_a(z_0; A_0).      \lb{6.46} 
\end{align}
Here the third equality in \eqref{6.46} is again justified in an analogous manner as 
\eqref{2.5j} (cf.\ \cite[Proposition~4.2.2]{GL09}).  
This shows the index formula in assertion $(ii)$. Clearly, if $z_0\in\rho(A_0)$ then the term $m_a(z_0; A_0)$ is absent and hence the index formula reduces to the one in assertion $(i)$.

The same argument as in the proof of Theorem~\ref{t5.5} shows that $z_0\in\rho(A_0)\cap\sigma_d(A_\Theta)$ is a zero of finite-type of the function
$[\Theta- M(\cdot)]$. Then \eqref{indinu2} follows from Theorem~\ref{t3.4a} and \eqref{coincide}.
\end{proof}

We conclude this section with the observation that the far simpler case of Krein-type 
resolvent formulas in terms of boundary data maps for one-dimensional Schr\"odinger and 
Sturm--Liouville operators discussed in \cite{CGM10} and \cite{CGNZ14}, readily yield 
analogous formulas for the index of these boundary data maps in terms of (algebraic) 
multiplicities of eigenvalues. One can follow the computation in the proof of Theorem~\ref{t6.4} 
line by line; we omit further details. 

\medskip
\noindent {\bf Acknowledgments.} 
F.G.\ gratefully acknowledges kind invitations to the Institute for Numerical Mathematics 
at the Graz University of Technology, Austria, and to the Department of Mathematical Sciences 
of the Norwegian University of Science and Technology, Trondheim, for parts of May and June 2015. The extraordinary hospitality by Jussi Behrndt and Helge Holden at each institution, as well as the stimulating atmosphere at both places, are greatly appreciated. 


\end{document}